\documentclass{mcom-l}
\usepackage{amsmath,lmodern,calc,tikz,textcomp}
\usetikzlibrary{tikzmark,calc}
\usetikzlibrary{decorations.pathreplacing}

\usepackage{amsmath}
\usepackage{amsfonts}
\usepackage{amssymb}
\usepackage{latexsym}
\usepackage{graphicx}
\usepackage{amsthm}
\usepackage{mathtools}
\usepackage{color}
\usepackage{todonotes}
\usepackage[T1]{fontenc}

\usepackage{color}
\newcommand{\uni}[1]{{\left\vert\kern-0.25ex\left\vert\kern-0.25ex\left\vert #1  \right\vert\kern-0.25ex\right\vert\kern-0.25ex\right\vert}} 
\newcommand{\unii}[1]{{\vert\kern-0.25ex\vert\kern-0.25ex\vert #1  \vert\kern-0.25ex\vert\kern-0.25ex\vert}} 
\newcommand{\Uniinv}[1]{{\big\vert\kern-0.25ex\big\vert\kern-0.25ex\big\vert #1  \big\vert\kern-0.25ex\big\vert\kern-0.25ex\big\vert}} 
\newcommand{\uniinv}[1]{{\left\vert\kern-0.25ex\left\vert\kern-0.25ex\left\vert #1  \right\vert\kern-0.25ex\right\vert\kern-0.25ex\right\vert}} 

\def\A{\mathcal{A}}
\def\P{\mathcal{P}}
\numberwithin{equation}{section}
\numberwithin{figure}{section}
 \def\Rtt{R_{2}}
\def\Uh{\widehat U}
\def\Vh{\widehat V}

\newcommand*{\Scale}[2][4]{\scalebox{#1}{\ensuremath{#2}}}%
\def\lh{\widehat\lambda}
\def\sh{\widehat\sigma}
\def\xh{\widehat x}
\def\xt{\widetilde x}
\def\Q{Q}
\def\Xt{\widetilde X}
\def\At{\widetilde A}

\def\Lh{\widehat \Lambda}
\def\Ut{\widetilde U}
\def\Vt{\widetilde V}
\def\St{\widetilde \Sigma}
\def\Sh{\widehat \Sigma}

\def\Xh{\widehat X}
\def\Xt{\widetilde X}

\def\ut{\widetilde u}
\def\uh{\widehat u}

\def\ggap{{\rm gap}_c}
\def\gap{{\rm gap}}
\def\Gap{{\rm Gap}}

\newtheorem{theorem}{Theorem}[section]

\newtheorem{remark}{{\sc Remark}}[section]

\newcommand{\ignore}[1]{}

\title[Ritz vectors and singular vectors]{
Sharp error bounds for Ritz vectors 
and approximate singular vectors 
}
\author{Yuji Nakatsukasa}
\address{
Mathematical Institute, University of Oxford, Oxford, OX2 6GG, UK.
}

\email{nakatsukasa@maths.ox.ac.uk}
\thanks{This work was supported by JSPS grants No. 17H01699 and 18H05837, and JST grant JPMJCR1914.}



\begin{document}
\maketitle

\begin{abstract}
  We derive sharp bounds for the accuracy of approximate
  eigenvectors (Ritz vectors) obtained by the Rayleigh-Ritz
  process for symmetric eigenvalue problems.
Using information that is available or easy to estimate, 
our bounds improve the classical Davis-Kahan $\sin\theta$ theorem
by a factor that can be arbitrarily large, and can give 
nontrivial information even when the $\sin\theta$ theorem suggests that a Ritz vector 
might have no accuracy at all. 
We also present extensions in three directions, deriving 
error bounds for 
 invariant subspaces, 
 singular vectors and subspaces 
computed by a (Petrov-Galerkin) projection SVD method, and 
 eigenvectors of self-adjoint operators on a Hilbert space. 
\end{abstract}


\section{Introduction}
It is well known that the eigenvector
corresponding to 
a near-multiple eigenvalue is ill-conditioned. 
Specifically, 
the classical Davis-Kahan theory~\cite{daviskahan} implies that the condition number of eigenvectors of symmetric or Hermitian matrices is $1/\gap$, where $\gap$ is the smallest distance between the particular eigenvalue and the other eigenvalues. For example, if $(\lh,\xh)$ with $\|\xh\|=1$ is an approximation to an exact eigenpair $(\lambda,x)$ of a symmetric matrix $A$ 
with residual $\|r\|=\|A\xh- \lh\xh\|$, then the Davis-Kahan $\sin\theta$ theorem gives the error bound for 
$\xh$~\cite{daviskahan},\cite[Ch.~11]{parlettsym}:
\begin{equation}  \label{eq:classical}
\sin\angle(x,\xh)\leq \frac{\|r\|}{\ggap},  
\end{equation}
 where 
$\angle(x,\xh) = \mbox{acos}\frac{|\xh^Tx|}{\|\xh\|\|x\|}$ and 
 $\ggap$ is the distance between $\lh$ and the eigenvalues of $A$ other than $\lambda$ (where the subscript stands for ``classical''). 
Here and throughout, $\|\cdot\|$ for vectors denotes the standard Euclidean norm. 
In view of the bound~\eqref{eq:classical}, it is commonly believed that if $\ggap$ is smaller than the residual $\|r\|$, then we cannot guarantee any accuracy in the computed eigenvector $\xh$.

In this work, 
we partly challenge this belief. 
Namely, we examine the accuracy of eigenvectors obtained by the Rayleigh-Ritz process (R-R), the most widely-used process for computing partial (usually extremal) eigenpairs of large-scale symmetric/Hermitian matrices, and show that~\eqref{eq:classical} can be improved---often significantly, and by a factor that can be arbitrarily large---using quantities that are readily available (or can be estimated cheaply) after the computation.

Of course, the classical Davis-Kahan bound is tight in general: In the absence of additional information other than $\|r\|$ and $\ggap$, we cannot improve~\eqref{eq:classical}, in that 
there exist examples for which the bound~\eqref{eq:classical} is essentially tight. 
However, when $(\lh,\xh)$ is a computed approximate eigenpair (Ritz pair) 
obtained by R-R, there is usually abundant additional information available that~\eqref{eq:classical} does not use: most importantly, the residual $r$ is orthogonal to the trial subspace, which is rich in the eigenspace corresponding to not only $\lambda$ but also eigenvalues close to $\lh$. 
Moreover, since the trial subspace in R-R usually contains approximation to 
nearby eigenpairs (e.g. when looking for the smallest eigenvalues), 
a bound can be computed for $\Gap$ (which we call the ``big Gap''), 
which is roughly the distance between the Ritz value $\lh$ and eigenvalues not approximated by the Ritz values; see~\eqref{eq:Gapgapdef} for the precise definition. 
These are the crucial properties that allow us to improve the Davis-Kahan bound~\eqref{eq:classical}---in other words, we take into account the matrix structure generated automatically by R-R to derive sharp bounds for the Ritz vector error. 

Our results essentially show that up to a modest constant, the $\ggap$ in \eqref{eq:classical} can be replaced by $\Gap$, which is usually much wider, thus improving classical results. 
Another way to understand our results is via (structured) perturbation theory: while an eigenvector has condition number $1/\ggap$ if a general perturbation is allowed, R-R imposes a structure in the perturbation that reduces the \emph{structured} condition number to $1/\Gap$. 

Qualitatively speaking, the fact that the accuracy of Ritz vectors depends on $\Gap$ rather than $\ggap$ was pointed out by Ovtchinnikov~\cite[Thm.~4]{ovtchinnikov2006cluster}. However, the bounds there involve quantities that are unavailable and diffucult to estimate, such as the projector onto an exact eigenspace. 
 Our bounds  are easy to compute or estimate, using information that is available after a typical computation of an approximate eigenpairs via the R-R process. Our bounds are also tight, in that they cannot be improved without additional information. 

In addition, we extend the results in three ways. 
First, we obtain error bounds for invariant subspaces 
(spanned by more than one eigenvector)
computed by R-R. This gives an answer to one of the open problems suggested in Davis-Kahan's classical paper. 
Second, we derive their SVD variants, establishing tight bounds for the quality of approximate singular vectors and singular subspaces associated with  the largest singular values, obtained by a (Petrov-Galerkin) projection method. 
Finally, we generalize the error bounds to eigenvectors of self-adjoint operators on a Hilbert space. 

\emph{Notation.}
$\lambda(A)$ denotes the spectrum (set of eigenvalues) of a symmetric matrix $A$. 
 $\sigma(A)=\{\sigma_i(A)\}_{i=1}^{\min(m,n)}$ is the set of singular values of $A\in\mathbb{C}^{m\times n}$, where $\sigma_1(A)\geq \sigma_2(A)\geq \cdots \geq \sigma_{\min}(A)=\sigma_{\min(m,n)}(A)\geq 0$. 
$I_n$ denotes the $n\times n$ identity matrix. 
$Q_\perp\in\mathbb{C}^{n\times (n-k)}$ is the orthogonal complement of 
$Q\in\mathbb{C}^{n\times k}$. 
Quantities involved in the R-R process wear a hat (e.g. $\lh,\xh,\Xh$), and those with tildes are auxiliary objects for the analysis. 
Norms
 $\|\cdot\|$ without subscripts denote the spectral norm equal to the largest singular value, which for vectors are the Euclidean norm.
 $\uniinv{\cdot}$ for inequalities that hold for any fixed unitarily invariant norm. Inequalities involving $\|\cdot\|_{2,F}$ hold for the spectral and Frobenius norms, but not necessarily for any unitarily invariant norm. 
We denote by $\mathcal{A}$ a self-adjoint operator on a Hilbert space, $\lambda(\mathcal{A})$ its spectrum, and $\|\mathcal{A}\|$ its spectral (operator) norm.
We drop the subscript $i$ in $\lh_i,\lambda_i$ when this can be done without causing confusion. 
We always normalize eigenvectors and Ritz vectors to have unit norm $\|x\|=\|\xh\|=1$. 

Unless otherwise stated, for definiteness we assume that the Ritz values $\lh_1,\ldots\lh_k$ approximate the smallest eigenvalues of $A$ (and accordingly the Ritz values are arranged in increasing order $\lh_1\leq \lh_2\leq \cdots \leq\lh_k$). This is a typical situation in applications, and clearly the discussion covers the case where the largest eigenvalues are sought (if necessary by working with $-A$). A less common but still important case is when interior eigenvalues are desired, for example those lying in an interval~(e.g.~\cite{li2016thick}). Our results are applicable to this case also; one subtlety here is that some care is needed in estimating $\Gap_i$, since the Ritz values tend to contain outliers in this case. 

\section{Setup}\label{sec:setup}
\subsection{Big (good) $\Gap$, small (bad) $\gap$}\label{sec:RR}
Let $A\in\mathbb{C}^{n\times n}$ be the (large) Hermitian matrix whose partial eigenvalues are sought, and let $\Q\in\mathbb{C}^{n\times k} (n\geq k$, usually $n\gg k)$ be a trial subspace with orthonormal columns $\Q^*\Q=I_k$
(obtained e.g. via Lanczos, LOBPCG, Jacobi-Davidson or the generalized Davidson method~\cite{baitemplates}). Following standard practice, 
for a matrix with orthonormal columns $\Q$, we identify the matrix $\Q$ with its column space $\mbox{Span}(\Q)$. 
R-R obtains approximate eigenvalues (Ritz values) and eigenvectors 
(Ritz vectors) as follows. 
\begin{enumerate}
\item Compute the $k\times k$ matrix $\Q^{\ast}A\Q$. 
\item Compute the eigendecomposition 
 $\Q^{\ast}A\Q=\Omega\Lh \Omega^\ast$. 
$(\lh_1,\ldots,\lh_k)=\mbox{diag}(\Lh)$ are the Ritz values, and 
$\Xh:=[\xh_1,\ldots,\xh_k] = Q\Omega$ are the Ritz vectors. 
\end{enumerate}
The Ritz pairs $(\widehat\lambda_i,\widehat x_i)$ thus obtained satisfy
 $\widehat x_i\in \mbox{span}(\Q)$ for all $i$, and 
since $Q^*(AQ\Omega-Q\Omega\widehat\Lambda)=Q^*AQ\Omega-\Omega\widehat\Lambda=
\Omega\Lh \Omega^\ast\Omega-\Omega\widehat\Lambda=0$ by construction, we have---crucially for this work---the orthogonality between $Q$ and  the residuals $A\xh_i-\lh_i\xh_i\perp \Q$, for every $1\leq i\leq k$. 
Throughout we assume $k\geq 2$; indeed when $k=1$ there is no room for improvement upon Davis-Kahan.


Underlying R-R is a matrix of particular structure: 
Let $Q_\perp$ be the orthogonal complement of $Q$, such that 
  $[Q\ Q_\perp]$ is a square unitary matrix (and hence so is  $[Q\Omega, Q_\perp]$), and consider the unitary transformation applied to $A$
\begin{equation} \label{eq:At}
\At:=
  [Q\Omega,Q_\perp]^*A[Q\Omega,Q_\perp]= \begin{bmatrix}
\quad 
    \begin{matrix}
     \lh_1 && \\
       &\ddots& \\
      &&\lh_k            
    \end{matrix}\quad 
\begin{matrix}
\mbox{---}r_1^T\mbox{---}\\
 \vdots  \\
\mbox{---}r_k^T\mbox{---}
\end{matrix}\quad \quad 
\\
 \begin{matrix}
|& &| \\  
r_1&\cdots &r_k
\\
|& &|
 \end{matrix} \begin{matrix}
& & \\  & \Scale[2.5]{A_3}
\\& 
 \end{matrix}
\end{bmatrix}.
\end{equation} 

Here $R=(Q_\perp)^* AQ\Omega=[r_1,r_2,\ldots,r_k]$; we use the subscript $3$ in $A_3$ because later we partition the $(1,1)$ block further into two pieces. 

Suppose $(\lambda_i,\xt_i)$ is an exact eigenpair of $\At$ such that $\At\xt_i=\lambda_i\xt_i$. 
Denote by $z_i\in\mathbb{C}^{n-k}$ the vector of the bottom $n-k$ elements of $\xt_i$, and by $w_i$ the $i$th element, and by $y_i\in\mathbb{C}^{k-1}$ the first $k$ elements of $\xt_i$, except the $i$th.
For example when $i=1$, we have $\xt_1=\left[\begin{smallmatrix}w_1\\ y_1\\z_1\end{smallmatrix}\right]$.
Then since $x_i=[\widehat X,Q_\perp]\xt_i$ is the corresponding eigenvector of $A$, and $(\lh_i,\xh_i)$ is a Ritz pair with 
$\xh_i =[\widehat X,Q_\perp]e_i$ where $e_i$ is the $i$th column of identity $I_n$, it follows that 
$\cos\angle(x_i,\xh_i)=|e_i^T\xt_i|=|w_i|$, and hence 
\begin{equation}
  \label{eq:sinx1xh1}
\sin\angle(x_i,\xh_i)=\sqrt{\|y_i\|^2+\|z_i\|^2}.  
\end{equation}
 This is a key fact in the forthcoming analysis. 

Fundamental in this work is the distinction between the ``big gap'' $\Gap_i$ and the ``small gap'' $\gap_i$, defined for $i=1,\ldots,k$ by 
\begin{equation}
  \label{eq:Gapgapdef}
\Gap_i:=\min|\lambda_i-\lambda(A_3)|  ,\qquad 
\gap_i:=\min_{j\in\{1,\ldots,k\}\backslash i}|\lambda_i-\lh_j|.
\end{equation}
Intuitively, $\Gap_i$ measures the distance between the target $\lambda_i$ and the undesired eigenvalues, whereas $\gap_i$ is that between $\lambda_i$ and all the other eigenvalues\footnote{\label{foot1}There is a subtle difference between $\Gap_i,\gap_i$ in~\eqref{eq:Gapgapdef} and $\ggap$ in~\eqref{eq:classical}, in that the former is the difference between an exact desired eigenvalue and approximate undesired eigenvalues, whereas $\ggap$ is the opposite. 
Our forthcoming analysis uses~\eqref{eq:Gapgapdef}, and 
we return to this difference in Remark~\ref{rem:gapdef}.
Also note that 
we do not explicitly assume that $\Gap_i\geq \gap_i$; the bounds continue to hold without such assumptions, although when they do not hold, the bounds can be worse than Davis-Kahan.}, 
including the desired ones (e.g, $\lambda_{2}$). 
For example when $R\rightarrow 0$, we have $\Gap_i\rightarrow\min|\lambda_i-\lambda_{k+1}|$; 
by contrast $\gap_i \rightarrow \min(|\lambda_i-\lambda_{i+1}|,|\lambda_i-\lambda_{i-1}|)$. Observe that 
$\Gap_i\geq \gap_i$, and we typically have $\Gap_i\gg \gap_i$. 
We illustrate this in Figure~\ref{fig:gapGap} for $i=1$. Throughout the paper, it is helpful to consider the case $i=1$, where the target eigenpair is the smallest one\footnote{When the target eigenpairs are the smallest or largest ones, it also becomes easier to obtain lower bounds for $\gap_1$ and $\Gap_i$. For example, since the Courant-Fisher minimax theorem $\lambda_i\leq \lh_i$, we have $\gap_1\geq \min_{j\in\{1,\ldots,k\}\backslash 1}|\lh_1-\lh_j|$, which involves only computed quantities. Similarly, we have $\Gap_i\geq \min|\lh_i-\lambda(A_3)|$.}. 
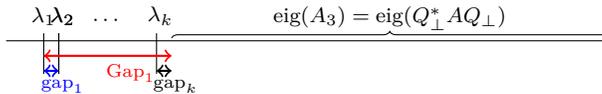
\begin{figure}[htbp]
  \centering
\footnotesize
\begin{tikzpicture}[decoration=brace]
\draw (0,0) -- (8,0);
\draw (.5,-0.44) -- (.5,0.1) node[anchor=south] {$\lambda_1$};
\draw (.7,-0.44) -- (.7,0.1) node[anchor=south] {\ $\lambda_2$};
\draw (.7,-0.44) -- (.7,0.1) node[anchor=south] {\ $\lambda_2$};
\draw (1.3,0.1) node[anchor=south] {\ $\cdots$};
\draw (2,-0.44) -- (2,0.1) node[anchor=south] {\ $\lambda_k$};
\draw[decorate, yshift=0.3ex]  (2.2,0) -- node[above=0ex] {$\mbox{eig}(A_3)
=\mbox{eig}(Q_\perp^*AQ_\perp)
$}  (8,0);
\draw[blue,thick,<->] (.5,-0.4) -- (.7,-0.4) node[anchor=north] {\ \scriptsize $\mbox{gap}_1$};
\draw[red,thick,<->] (.5,-0.2) -- (2.2,-0.2) node[anchor=north,xshift=-4.5ex] {\scriptsize $\mbox{Gap}_1$};
\draw[thick,<->] (2,-0.4) -- (2.2,-0.4) node[anchor=north] {\ \scriptsize $\mbox{gap}_k$};
\end{tikzpicture}
\normalsize 
  \caption{Illustration of typical situation when the smallest eigenvalues are sought, and $R$ is small enough so that $\lh_i\approx \lambda_i$. While the small gap is 
$\gap_i=\min_{j\neq i}|\lh_i-\lambda_j|\approx \min_{j\neq i}|\lh_i-\lh_j|$, the big Gap is much bigger
$\Gap_i=\min|\lambda_i-\lambda(A_3)|$. 
}
  \label{fig:gapGap}
\end{figure}

In addition to $\gap_i$ and $\Gap_i$, some of the bounds we derive involve $|\lambda_i-\lh_j|$ for a fixed $j\in\{1,\ldots,k\}\backslash i$. These lie between $\gap_i$ and $\Gap_i$. 

Recalling~\eqref{eq:At}, the information clearly available after R-R are the Ritz pairs $(\lh_i,\widehat x_i)$ for $i=1,\ldots, k$ and the norms of the individual $r_i$, 
 because they are equal to the residuals $\|r_i\|=\|A\widehat x_i-\lh_i \widehat x_i\|$. 
In addition, one can reasonably expect that an estimate (or better yet, a lower bound) is available for $\Gap_i$ for each $i$, or at least for small $i$: when 
the smallest $k$ eigenpairs are sought, 
the trial subspace $\Q$---assuming it has been chosen appropriately by the algorithm used---is
 expected to be rich in the eigenspace corresponding to those eigenvalues. 
 It then follows from standard eigenvalue perturbation theory that $A_3$ contains only eigenvalues that are roughly at least as large as $\lambda_{k+1}(A)$ (up to $\|R\|$, or indeed~$\frac{\|R\|^2}{\gap}$~\cite{rcli05}).
Therefore, although the exact value of $\lambda_{k+1}(A)$ is unknown, we can use the knowledge of the Ritz values $\lh_i$ to estimate $\Gap_i$, for example $\Gap_i \approx |\lh_i-\lh_{k+1}|$ or 
$\Gap_i \gtrsim |\lh_i-\lh_{k}|$; we use the latter, approximate lower bound in our experiments. 
Similarly, one can estimate $\gap_i$ for example as $\gap_i\approx \min_{j\neq i}|\lh_i-\lh_j|$. 

In practice, an important feature of the residuals is that they are typically graded: $\|r_1\|\ll\|r_2\|\ll\cdots\ll \|r_k\|$. This is because the extremal eigenvalues converge much faster than interior ones; a fact deeply connected with polynomial (and rational) approximation theory~\cite[\S~33]{trefbau}. We derive bounds (e.g. Theorem~\ref{thm:eigrefine}) that respect this property, and hence give sharp bounds in practical situations.

We note that previous bounds exist that involve the big $\Gap_i$ rather than $\gap_i$; most notably 
(aside from Ovtchinnikov's result~\cite{ovtchinnikov2006cluster} mentioned in the introduction) Davis-Kahan's generalized $\sin\theta$ theorem  where the angles between subspaces of different dimensions are bounded~\cite[Thm.~6.1]{daviskahan}. In this case, however, (in addition to comparing e.g. a vector and a subspace rather than two vectors) the numerator is replaced by the entire $\|R\|$ rather than the $i$th column $\|r_i\|$. 
The bounds we derive  
 essentially show that, up to a small constant, (i) the small $\gap_i$ in~\eqref{eq:classical} can be replaced by the big $\Gap_i$, and (ii) the 
numerator is the $i$th column $\|r_i\|$. 
These combined give a massively improved error bound for $\xh_i$, especially for small values of $i$. The next section illustrates the first aspect, and the second will be covered in Section~\ref{sec:eigsharp}. 

\section{$2\times 2$ partitioning}\label{sec:2by2}
We will derive three error bounds for Ritz vectors; the first, obtained in this section, is simple and vividly illustrates the roles of $\gap$ and $\Gap$, but not sharp in a practical setting. 
In Section~\ref{sec:eigsharp} we derive two more bounds that give better bounds in practice. 

Here we consider a simplified $2\times 2$ block partitioning  of~\eqref{eq:At}  where 
\begin{equation}
  \label{eq:Aef}
\At (=
  [Q\Omega,Q_\perp]^*A[Q\Omega,Q_\perp])=
\begin{bmatrix}
  \Lh_1 & R^*\\R& A_3
\end{bmatrix},   
\end{equation}
where $\Lh_1=\mbox{diag}(\lh_1,\Lh_2)\in\mathbb{R}^{k\times k}$ and $\|R\|=\|A\Xh-\widehat X\Lh_1\|$ are the computed quantities. In other words, we do not distinguish the columns $r_i$ of $R$ but treat $\|R\|$ as a single residual term. 


Below, we derive bounds for $\sin\angle(x_i,\xh_i)$ applicable to $i=1,\ldots,k$. In our analysis, we assume that $\lh_i$ is the $(1,1)$ element of $\At$. 
This simplifies the discussion and loses no generality as we can permute the leading $k\times k$ block of $\At$. 
Moreover, we drop the subscript $i$ in the remainder of this section for simplicity.


\begin{theorem}\label{thm:mainscalar}
Let $A$ be a Hermitian matrix as in~\eqref{eq:At}, 
for which $(\lh,\xh)$ is a Ritz pair with $\lh=\lh_1$. 
Let $(\lambda,x)$ be an eigenvector of $A$, and 
let $\Gap=\min|\lambda-\lambda(A_3)|$ and 
$\gap=\min|\lambda-\lambda(\Lh_2)|$. 
Then writing $R_{2}=[r_2,r_3,\ldots,r_k]$, we have 
\begin{equation}  \label{eq:boundvec}
\sin\angle(x,\xh )\leq 
\frac{\|R\|}{\Gap}\sqrt{1+\frac{\|R_{2}\|^2}{\gap^2}}
\quad
\left(
\leq  \frac{\|R\|}{\Gap}(1+\frac{\|R_{2}\|}{\gap})
\right). 
\end{equation}
\end{theorem}
Note that clearly $\|R_{2}\|\leq \|R\|$, so 
the result implies 
$\sin\angle(x,\xh )\leq 
\frac{\|R\|}{\Gap}\sqrt{1+\frac{\|R\|^2}{\gap^2}}
\leq  \frac{\|R\|}{\Gap}(1+\frac{\|R\|}{\gap}). 
$
\begin{proof}  
Let 
$\xt=\begin{bmatrix} w\\ y\\z\end{bmatrix}$ be an eigenvector of $\At$ 
as in~\eqref{eq:Aef}
such that $\At
\begin{bmatrix}
 w\\ y\\z
\end{bmatrix}=\lambda 
\begin{bmatrix}
 w\\ y\\z
\end{bmatrix}$, with $w\in\mathbb{C},y\in\mathbb{C}^{k-1}$. 
Then 
since $\sin\angle(x,\xh ) = 
\left\|
\begin{bmatrix}
y\\z
\end{bmatrix}
\right\|$ from~\eqref{eq:sinx1xh1}, the goal is to bound 
$\|y\|$ and $\|z\|$. 
The bottom part of $\At\xt=\lambda \xt$ gives 
\[
(\lambda I_{n-k}-A_3) z 
=R\begin{bmatrix}
  w\\y
\end{bmatrix}
, 
\]
from which we obtain 
\begin{equation}  \label{eq:dkspecial}
\|z\|\leq \|(\lambda I_{n-k}-A_3)^{-1}\|
\left\|R\begin{bmatrix}
  w\\y
\end{bmatrix}
\right\|
\leq \|(\lambda I_{n-k}-A_3)^{-1}\|\|R\| = \frac{\|R\|}{\Gap}  . 
\end{equation}
Note that the denominator is $\Gap$, not $\gap$. 
We also note that the final bound is Davis-Kahan's generalized $\sin\theta$ theorem where subspaces of different sizes ($\xh$ and the $[x_1,\ldots,x_k]$) are compared (and when the perturbation is off-diagonal); in fact, we can also obtain $\|z\|\leq \frac{\|R\|}{\Gap} 
\left\|
\big[
\begin{smallmatrix}
  w\\y  
\end{smallmatrix}
\big]
\right\|$, which is the generalized $\tan\theta$ theorem. 

From the second block of
$\At\xt=\lambda \xt$  we have 
\[
(\lambda I_{k-1} -\Lh_{2})y = R_{2}^*z,
\]
and since $\|(\lambda I_{k-1} -\Lh_{2})y\|\geq 
\gap \|y\|$, 
 we obtain the important bound 
\begin{equation}
  \label{eq:ybound1}
\|y\|
\leq \frac{\|R_{2}\|}{\gap}\|z\|.  
\end{equation}
Combining with~\eqref{eq:dkspecial} we obtain 
\begin{equation}
  \label{eq:y2knya}
\|y\|
\leq \frac{\|R\|\|R_{2}\|}{\gap\cdot \Gap} . 
\end{equation}
Therefore, 
we conclude that 
\[
\sin\angle(x,\xh )^2 = 
\left\|
\begin{bmatrix}
y\\z
\end{bmatrix}
\right\|^2 \leq  
\frac{\|R\|^2}{\Gap^2}\left(
1+\frac{\|R_{2}\|^2}{\gap^2}
\right), 
\]
giving~\eqref{eq:boundvec}. 
\end{proof}

\ignore{
Let us examine how well the bound~\eqref{eq:boundvec} predicts the experimental observations. 
\begin{itemize}
\item if $\|R\|\leq \gap$, then~\eqref{eq:boundvec} gives $\angle(x,\xh )=O(\|R\|/\Gap)$; in particular, $\angle(x,\xh )\leq \sqrt{2}\|R\|/\Gap$. 
\item if $\|R\|\geq \gap$, then~\eqref{eq:boundvec} gives  $\angle(x,\xh )= O(\|R\|^2/(\gap\cdot \Gap))$. 
\end{itemize}
These both accurately reflect the experiments, see also Section~\ref{sec:exp}. 
}

We make several remarks regarding the theorem. 

\begin{remark}[Qualitative behavior of bounds]
Theorem~\ref{thm:mainscalar} shows that 
\begin{itemize}
\item if $\|R\|\leq \gap$, then $\sin\angle(x,\xh )\lesssim \frac{\|R\|}{\Gap}$. 
\item if $\|R\|\geq \gap$, then $\sin\angle(x,\xh )\lesssim \frac{\|R\|^2}{\gap\cdot \Gap}$. 
\end{itemize}
Note how Davis-Kahan's bound is insufficient to explain these: 
When $\|R\|\leq \gap$, we improve the bound~\eqref{eq:classical} 
by a factor $\gap/\Gap$, which is typically $\ll 1$. 
Moreover, when $\|R\|\geq \gap$, classical results suggest $\xh $ may have no accuracy at all. Nonetheless, Theorem~\ref{thm:mainscalar} shows that there is still 
a nontrivial bound for $\sin\angle(x,\xh )$  as long as $\|R\|\lesssim \sqrt{\gap\cdot \Gap} (\gg \gap)$. 
These results are particularly relevant when 
only low-accuracy solutions are available, so that $\|R\|$ is much larger than working precision. 
\end{remark}
\begin{remark}[Effect of finite precision arithmetic]\label{rem:finite}
Crucial in the above argument is that $\Lh_1$ has zero off-diagonal elements. 
In practice in finite-precision arithmetic, 
the Rayleigh-Ritz process inevitably results in 
$\Lh_1$ in~\eqref{eq:Aef} with 
off-diagonal elements that are $O(u)$ instead of $0$, due to roundoff errors (assuming for simplicity $\|A\|=O(1)$). It is therefore important to address how they affect the bounds. As mentioned in the introduction, classical perturbation theory shows that these $O(u)$ terms will perturb $\xh $ by up to $O(u/\gap)$. Since the off-diagonal $O(u)$ elements in $\Lh_1$ indeed lie in the directions that perturb the eigenvector the most (we return to this in Section~\ref{sec:eigpert}), to account for roundoff errors we will need to add the term $O(u/\gap)$ to the bound~\eqref{eq:boundvec}. 
This remark becomes important especially when $\|R\|$ is small, 
so that $O(u/\gap)$ is not negligible relative to $\frac{\|R\|}{\Gap}$. 
In other words, the folklore that eigenvectors cannot be computed with precision higher than $u/\gap$ is true; 
what we refute is the belief that the bound $\|R\|/\gap$ 
(or $\|r\|/\gap$)
is sharp---our result shows that when $\|R\|> \gap$ but $\|R\|< \Gap$, 
Rayleigh-Ritz computes eigenvectors of much higher accuracy than $\|R\|/\gap$. 
\end{remark}

\begin{remark}[Different partitionings]\label{rem:partition}
  We can obtain different bounds depending on where to partition, that is, we can invoke the bound~\eqref{eq:boundvec} by taking a $k\leftarrow k'$ for some $k'\leq k$. 
Each choice of $k'$ gives a different bound, since each gives different values of $\|R\|$ and $\Gap$ (along with $\gap$, though its dependence on $k'$ is usually much less significant). If the computational cost is not a concern, one can compute all possible partitionings and take the smallest bound obtained. However, the bounds in Section~\ref{sec:eigsharp} are often still  better in practice. 
\end{remark}

\begin{remark}[Proof via generalized Davis-Kahan and Saad]\label{rem:dksaad}
 The result~\eqref{eq:boundvec} can also be derived by combining (i) Saad's bound~\cite[Thm.~4.6]{saadbookeig}, which bounds $\sin\angle(\xh_i,x_i)$ relative to $\sin\angle(Q,x_i)$, the angle  between the desired eigenvector and the trial subspace, and (ii) the generalized Davis-Kahan $\sin\theta$ theorem~\cite{daviskahan}, in which two subspaces of different dimensions are compared. 
 Here we presented a first-principles derivation, as we use the same line of arguments to derive improved and generalized bounds in the forthcoming sections. 
Also noteworthy is Knyazev's paper~\cite{knyazevmc97}, which
generalizes Saad's bound to subspaces. He also shows that Ritz vectors contain quadratically small components in eigenvectors approximated by the other Ritz vectors. 
This is essentially captured in~\eqref{eq:ybound1}, which indicates $\|y\|=O(\|R\|^2)$ (absorbing the gaps in the constant).  We revisit this phenomenon for subspaces in Section~\ref{sec:RRangles}. 
\end{remark}

While we will not repeat them, Remarks~\ref{rem:finite} and \ref{rem:partition} are relevant throughout the paper. 
\subsection{Experiments}\label{sec:exp}
To illustrate Theorem~\ref{thm:mainscalar},
 we conduct the following experiment; throughout, all experiments were carried out in MATLAB version R2017a using IEEE double precision arithmetic with unit roundoff $\approx 1.1\times 10^{-16}$. 
Let 
\[
A =
\begin{bmatrix}
\Lh_1
& R^*\\R & A_3
\end{bmatrix}\in\mathbb{R}^{n\times n}
\]
where $n=10$ (the precise size of $n$ is insignificant), 
$\Lh_1 = -\big[  \begin{smallmatrix}
1+\gap& \\ & 1    
  \end{smallmatrix}\big]$ and 
 $A_3\succeq 0$, so that $\Gap\geq 1$. 
We take $R\in\mathbb{R}^{k\times 2}$ to be randomly generated matrices using MATLAB's {\tt randn} function, scaled so that $\|R\|$ is fixed to a value $10^{-i}$, for $i=0,\ldots,15$. For each $i$, we generate 100 such matrices $R$, and 
find the largest value of $\sin\angle(x,\xh )$ from the 100 runs (note that $\xh=[1,0,\ldots,0]^T$ by construction). These are shown as 'observed' in Figure~\ref{fig:exp}, along with (i) the classical bound $\|R\|/\gap$, (ii) the new bound~\eqref{eq:boundvec}, 
(iii) the bound $u/\gap$, in view of Remark~\ref{rem:finite}. 

In view of Remark~\ref{rem:finite}, 
 $\sin\angle(x,\xh )$ is bounded by the maximum of~\eqref{eq:boundvec} and (a small multiple of) $u/\gap$. Of course, we always have the trivial bound $\sin\angle(x,\xh )\leq 1$, so putting these together, we have the following bound in finite-precision arithmetic:
\begin{equation}
  \label{eq:finalbound}
\sin\angle(x,\xh )\leq \min\left(1,\max\big(O(\frac{u}{\gap}),\frac{\|R\|}{\Gap}\sqrt{1+\frac{\|R\|^2}{\gap^2}}\big)\right).
\end{equation}

We observe in Figure~\ref{fig:exp} that this is indeed the case, and the 
new bound~\eqref{eq:boundvec} gives remarkably sharp bounds for the observed values of $\sin\angle(x,\xh )$ (when it is not dominated by $u/\gap$, and gives a nontrivial bound $\leq 1$). This is despite the fact that we are plotting the looser bound $\frac{\|R\|}{\Gap}\sqrt{1+\frac{\|R\|^2}{\gap^2}}$ with $\|R_2\|$ in~\eqref{eq:boundvec} replaced by $\|R\|$, as using $\|R_2\|$ makes the bound depend on the particular random instance of $R$. 

As discussed above, the new bound~\eqref{eq:boundvec} has two asymptotic behaviors: $\approx \|R\|/\Gap$ when $\|R\|\leq \gap$, 
and $\approx \|R\|^2/(\gap\cdot \Gap)$ when $\|R\|\geq \gap$. This can be seen in the plots, as the change of slope in the new bound around $\|R\|\approx \gap$. From the plots with  $\gap=10^{-3}$ and $10^{-5}$, we see that this transition also reflects the observed values of $\sin\angle(x,\xh )$ quite accurately. In all cases, the classical Davis-Kahan bound $\|R\|/\gap$ tends to be severe overestimates (and $\|r\|/\gap$ as in~\eqref{eq:classical} is not much different), and the new bound can provide nontrivial information (bound smaller than 1) even when the Davis-Kahan  bound is useless with $\|R\|/\gap>1$, and the difference between Davis-Kahan and the new bound widens when $\gap$ is small. 

\begin{figure}[htpb]
\hspace{-4mm}
  \begin{minipage}[t]{0.49\hsize}
      \includegraphics[width=.95\textwidth]{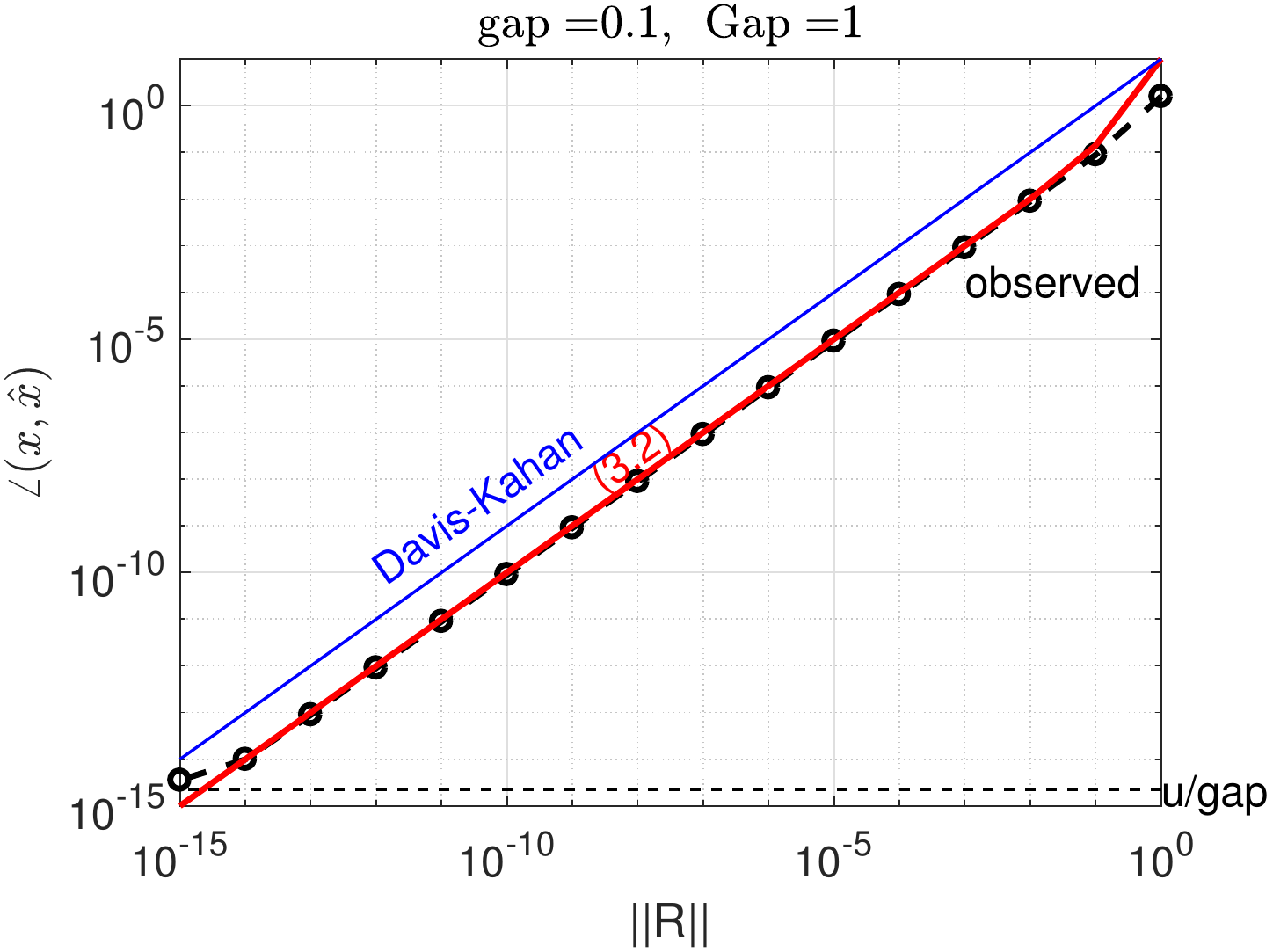}
  \end{minipage}   
  \begin{minipage}[t]{0.49\hsize}
      \includegraphics[width=.95\textwidth]{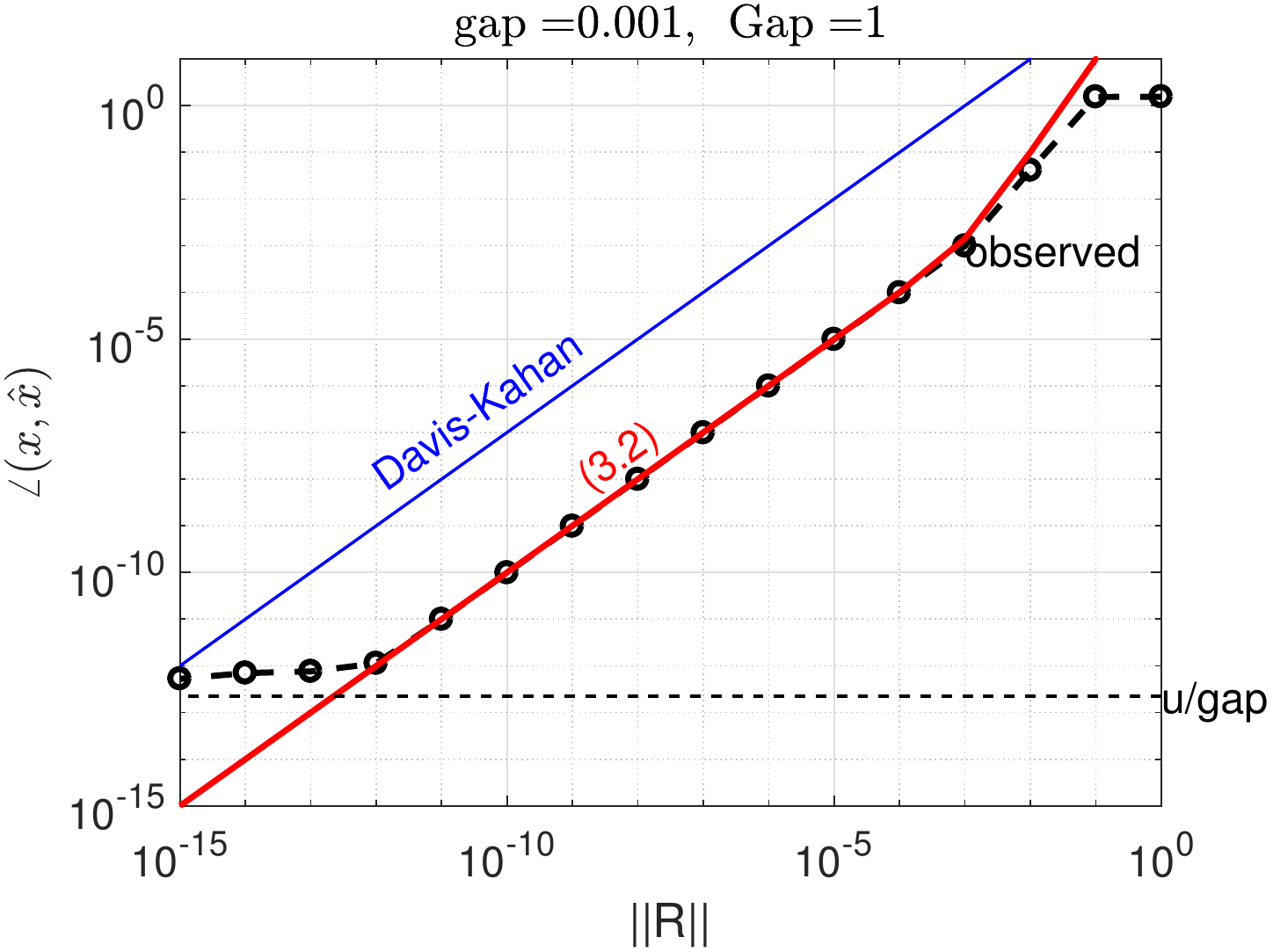}
  \end{minipage}\\
\vspace{2mm}

  \begin{minipage}[t]{0.49\hsize}
      \includegraphics[width=.95\textwidth]{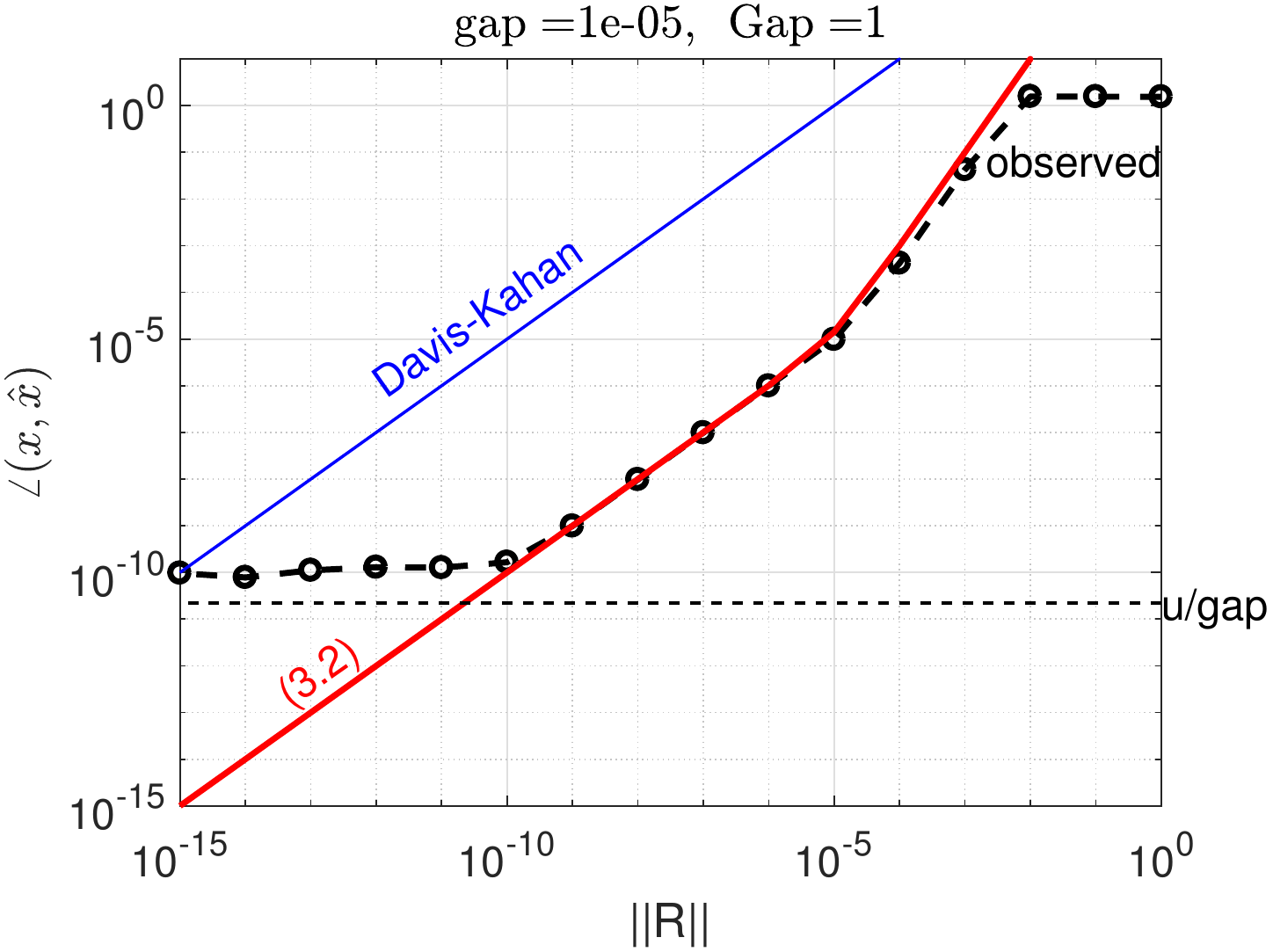}
  \end{minipage}
  \begin{minipage}[t]{0.49\hsize}
      \includegraphics[width=.95\textwidth]{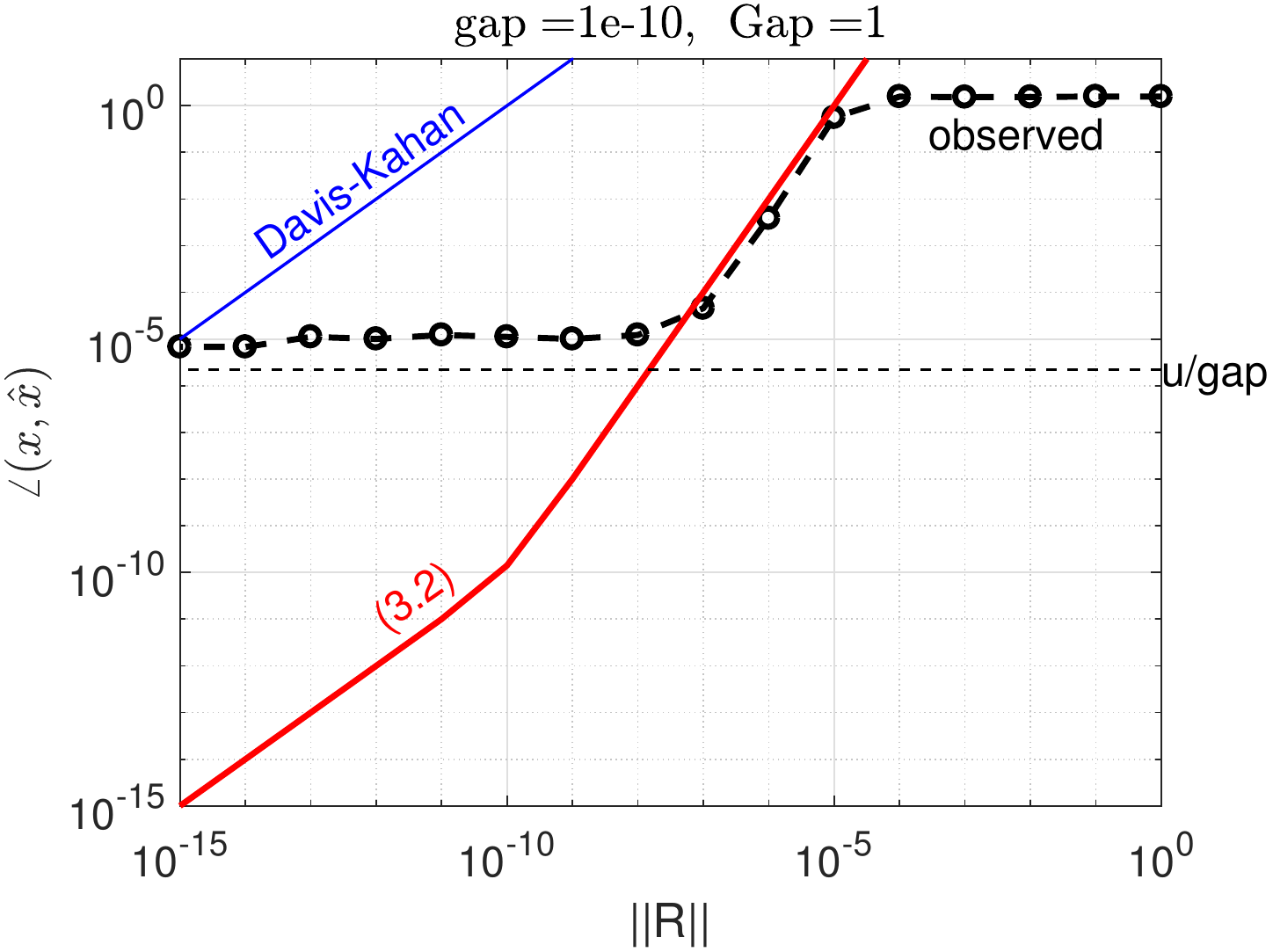}
  \end{minipage}
  \caption{Illustration of our bound~\eqref{eq:boundvec} (dashed red; with $R_2$ replaced by $R$), 
varying gap (upper-left: $\gap=10^{-1}$, upper-right: $\gap=10^{-3},$ 
lower-left: $\gap=10^{-5}$, lower-right: $\gap=10^{-10}$). 
Observe how sharp~\eqref{eq:boundvec} is, relative to the classical Davis-Kahan bound $\|R\|/\gap$. 
When $\|R\|\leq u/\gap$, 
the bound in finite-precision arithmetic would be the maximum between the new bound and $u/\gap$ (dashed black, constant line); see~\eqref{eq:finalbound}. 
}
  \label{fig:exp}
\end{figure}

\section{Improved error bounds for Ritz vectors}\label{sec:eigsharp}
\label{sec:improve}
The above experiments illustrate the sharpness of the bound~\eqref{eq:boundvec} given the information $\|R\|=\|[r_1,\ldots,r_k]\|$ and $\lh_i$, along with $\min(\mbox{eig}(A_3))$. 
When applied in practice, however, we find that the bound~\eqref{eq:boundvec} is usually a severe overestimate, as we illustrate in Section~\ref{sec:exp2}. The reason is that it does not distinguish $r_1$ from $r_k$ (say), while typically we have $\|r_1\|\ll \|r_k\|$, reflecting the difference in speed with which each Ritz pair converges, typically the extremal ones converging first. 

As noted in Section~\ref{sec:RR}, after R-R one also has information on the individual norms $\|r_i\|=\|A\xh_i-\lh_i\xh_i\|$. Here we derive bounds that are essentially sharp using all the information. 
We shall show that if $\|r_i\|$ are sufficiently small, 
then 
$\sin\angle(x,\xh )\lesssim \frac{\|r_1\|}{\Gap_1}$. This is 
usually a massive improvement over~\eqref{eq:boundvec}, and 
essentially sharp: we cannot improve the bound below $\frac{\|r_1\|}{\Gap_1}$. 
The argument is similar to Theorem~\ref{thm:mainscalar} but with more elaborate manipulations. The strategy is the same: bound $\|y\|$ in terms of $\|z\|$, and use this to bound 
$\left\|\big[
  \begin{smallmatrix}
    y\\z
  \end{smallmatrix}
\big]\right\|$. 

\begin{theorem}\label{thm:eigrefine}
In the setting of Theorem~\ref{thm:mainscalar}, 
\begin{itemize}
\item 
If $\Gap>\frac{\|\Rtt\|^2}{\gap}$, 
then 
\begin{equation}
  \label{eq:sin22}
\sin\angle(x,\xh ) 
\leq \frac{\|r_1\|}{\Gap-\frac{\|\Rtt\|^2}{\gap}}\sqrt{1+\frac{\|\Rtt\|^2}{\gap^2}}  . 
\end{equation}
\item If $\Gap>\sum_{i=2}^{k}\frac{\|r_{i}\|^2}{|\lambda  -\lh_{i}|}$, then 
\begin{equation}
  \label{eq:sin2indiv}
\sin\angle(x,\xh ) 
\leq 
\frac{\|r_1\|}{
\Gap-\sum_{i=2}^{k}\frac{\|r_{i}\|^2}{|\lambda  -\lh_{i}|}}
\sqrt{1+
\left(\sum_{i=2}^{k}\frac{\|r_{i}\|}{|\lambda  -\lh_{i}|}\right)^2}. 
\end{equation}
\end{itemize}
\end{theorem}
\begin{proof}

We first prove~\eqref{eq:sin22}. 
The main idea is to improve the bound~\eqref{eq:dkspecial} on $\|z\|$. As before we have 
$
(\lambda I_{k-1} -\Lh_{2})y = \Rtt^*z,$
so 
$\|y\|\leq \frac{\|\Rtt\|\|z\|}{\gap}$. We also have 
\[
(\lambda I_{n-k}-A_3) z =[r_1\ \Rtt]
\begin{bmatrix}
  w\\y
\end{bmatrix}. 
\]
This gives $(\lambda I_{n-k}-A_3) z -\Rtt y=r_1w$, hence 
 $\|(\lambda I_{n-k}-A_3) z\| -\|\Rtt y\|\leq \|r_1w\|$. 
Using $\|y\|\leq \frac{\|\Rtt\|\|z\|}{\gap}$ we obtain 
\[
\|(\lambda I_{n-k}-A_3) z\| -
\frac{\|\Rtt\|^2\|z\|}{\gap}
\leq \|r_1w\|. 
\]
Noting that 
 $\sigma_{\min}(\lambda I_{n-k}-A_3)=\Gap$, 
we have $\|(\lambda I_{n-k}-A_3) z\|\geq \Gap \|z\|$, hence 
\[
(\Gap-\frac{\|\Rtt\|^2}{\gap})\|z\| \leq \|r_1w\|. 
\]
Using the assumption $\Gap>\frac{\|\Rtt\|^2}{\gap}$ and the trivial bound $\|w\|\leq 1$ we obtain
\[
\|z\| \leq \frac{\|r_1\|}{\Gap-\frac{\|\Rtt\|^2}{\gap}}. 
\]
This 
together with $\|y\|\leq \frac{\|\Rtt\|\|z\|}{\gap}$ yields 
\[
\sin\angle(x,\xh ) = 
\left\|
\begin{bmatrix}
y\\z
\end{bmatrix}
\right\|
\leq \frac{\|r_1\|}{\Gap-\frac{\|\Rtt\|^2}{\gap}}\sqrt{1+\frac{\|\Rtt\|^2}{\gap^2}}, 
\]
giving~\eqref{eq:sin22}. 

The remaining task is to prove~\eqref{eq:sin2indiv}. 
The idea to improve the bound~\eqref{eq:ybound1} on $\|y\|$, or rather its individual entries, using 
\[
(\lambda I_{k-1} -\Lh_{2})y = \Rtt^*z. 
\]
Writing $y=[y_2,\ldots,y_k]^T$, 
the $i$th ($i=2,\ldots,k$) element gives
$(\lambda  -\lh_{i})y_{i} = r_{i}^*z, $
hence 
\begin{equation}  \label{eq:ithy}
|y_i|= \frac{|r_{i}^*z|}{|\lambda  -\lh_{i}|} \leq 
\frac{\|r_{i}\|\|z\|}{|\lambda  -\lh_{i}|},\quad i=2,\ldots, k. 
\end{equation}
We also have 
\[
(\lambda I_{n-k}-A_3) z =[r_1,r_2,\ldots,r_k]
\begin{bmatrix}
  w\\y_2\\\vdots\\y_{k}
\end{bmatrix}. 
\]
This gives $(\lambda I_{n-k}-A_3) z -\Rtt y=r_1w$, and 
\begin{equation}
  \label{eq:rindiv}
\|(\lambda I_{n-k}-A_3) z\| -\sum_{i=2}^{k}\|r_{i}y_i\|\leq \|r_1w\|, 
\end{equation}
so using~\eqref{eq:ithy}
we obtain 
\[\|(\lambda I_{n-k}-A_3) z\| -
\sum_{i=2}^{k}\frac{\|r_{i}\|^2\|z\|}{|\lambda  -\lh_{i}|}
\leq \|r_1w\|  .\]
Again using $\|(\lambda I_{n-k}-A_3) z\|\geq\Gap\|z\|$, 
we therefore obtain 
\[
(\Gap-
\sum_{i=2}^{k}\frac{\|r_{i}\|^2}{|\lambda  -\lh_{i}|}
)\|z\| \leq \|r_1w\|. 
\]
Hence, using the assumption $\Gap>
\sum_{i=2}^{k}\frac{\|r_{i}\|^2}{|\lambda  -\lh_{i}|}$ and the trivial bound $\|w\|\leq 1$ we obtain
\[
\|z\| \leq \frac{\|r_1\|}{
\Gap-\sum_{i=2}^{k}\frac{\|r_{i}\|^2}{|\lambda  -\lh_{i}|}}.
\]
The fact $\sin\angle(x,\xh ) = 
\left\|
\begin{bmatrix}
y\\z
\end{bmatrix}
\right\|
$ together with~\eqref{eq:ithy} 
 completes the proof of \eqref{eq:sin2indiv}. 
\end{proof}

Note that since the bounds $\|y\|\leq \frac{\|\Rtt\|\|z\|}{\gap}$ and~\eqref{eq:ithy} are both valid, in both bounds~\eqref{eq:sin22} and~\eqref{eq:sin2indiv}, the term with the square root can be replaced with the minimum, that is, 
$\sqrt{1+\min\bigg(\frac{\|\Rtt\|^2}{\gap^2},\big(\sum_{i=2}^{k}\frac{\|r_{i}\|}{|\lambda  -\lh_{i}|}\big)^2\bigg)}$. This applies also to the bounds to follow, but for brevity we do not repeat this remark. 

We also note that the bounds~\eqref{eq:sin22} and \eqref{eq:sin2indiv} are not comparable. The bound \eqref{eq:sin22} involves the small $\gap$, which \eqref{eq:sin2indiv} avoids to some extent by using the individual residuals $\|r_i\|$; however, the heavy use of triangular inequalities in the bound \eqref{eq:rindiv} suggests \eqref{eq:sin22} can still be a significant overestimate. 
The ``sharpest'' bound one can obtain would be via directly bounding the norm $\|y\| = \|(\lambda I_{k-1} -\Lh_{2})^{-1}\Rtt^*z\|$. Nonetheless, experiments suggest~\eqref{eq:sin2indiv} is often a good bound, as we illustrate now. 


\subsection{Experiments}\label{sec:exp2}
We illustrate Theorem~\ref{thm:eigrefine} with experiments more practical than Section~\ref{sec:exp}. We let $A\in\mathbb{R}^{1000\times 1000}$ be the classical tridiagonal matrix with $2$ on the diagonal and -1 on the super- and sub-diagonals. This is a 1D Laplacian matrix, obtained by finite difference discretization. We then run the LOBPCG algorithm~\cite{lobpcg} to compute the smallest eigenpair with a random initial guess, working with a $k=50$-dimensional subspace. 

Figure~\ref{fig:lob} (left) shows the convergence of $\sin\angle(\xh_1,x_1)$ along with four bounds: Davis-Kahan's $\sin\angle(\xh_1,x_1)\leq \frac{\|r_1\|}{\gap}$, \eqref{eq:sin22}, \eqref{eq:sin2indiv} and \eqref{eq:boundvec} from the previous section. Some data are missing for 
\eqref{eq:sin22} and \eqref{eq:sin2indiv} in the early steps as they violated the assumption $\Gap>\frac{\|\Rtt\|^2}{\gap}$ or 
$\Gap>\sum_{i=2}^{k}\frac{\|r_{i}\|^2}{|\lambda  -\lh_{i}|}$; note that these assumptions can be checked inexpensively. 
To estimate $\Gap$ and $\gap$ 
we used the available quantities $\Gap \gtrsim \min|\lh_1-\lh_{k}|$,
$\gap \approx \min|\lh_1-\lh_{2}|$; the plots look nearly identical if the exact values are used. 

\begin{figure}[htbp]
  \begin{minipage}[t]{0.49\hsize}
\includegraphics[width=0.95\textwidth]{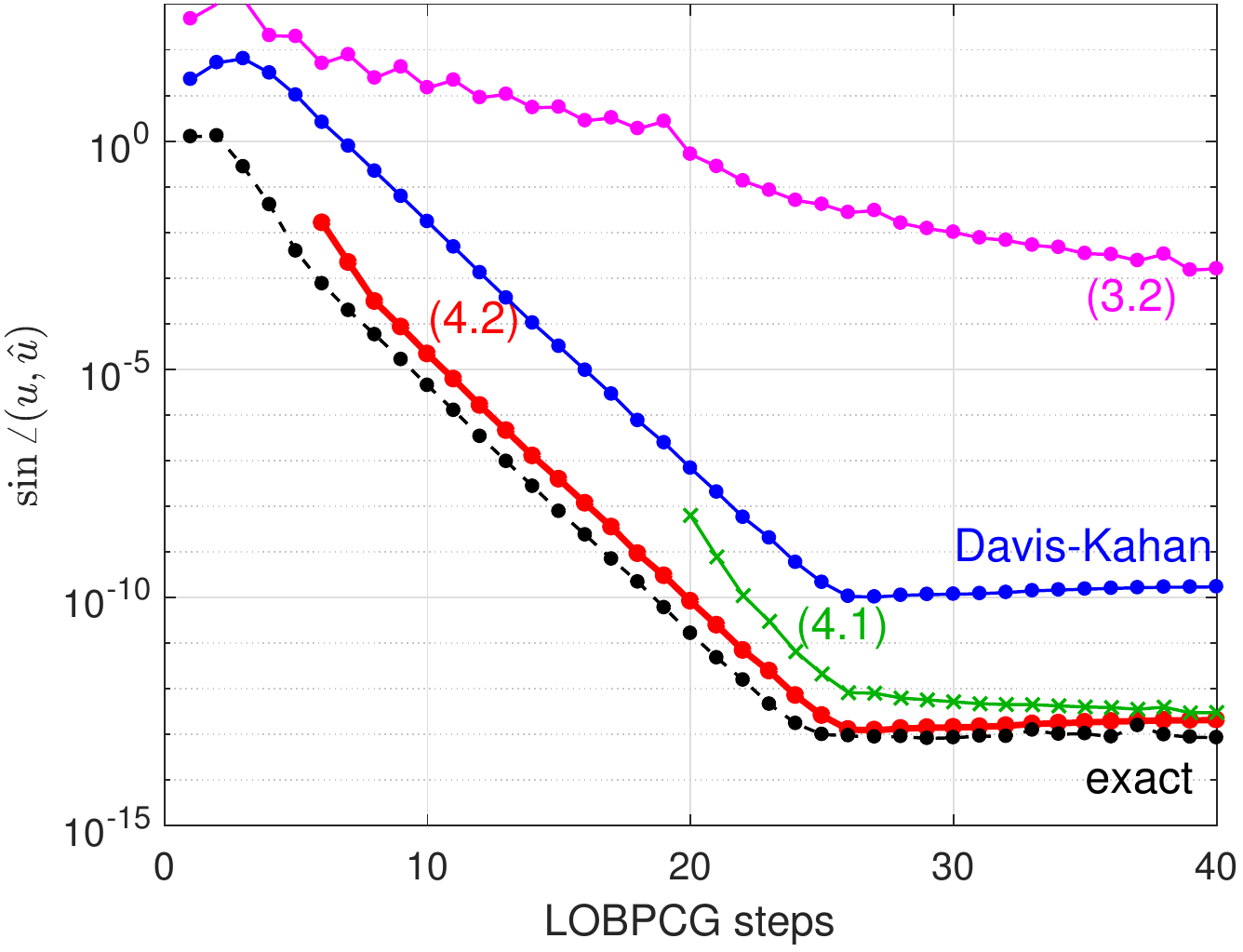}      
  \end{minipage}   
  \begin{minipage}[t]{0.49\hsize}
\includegraphics[width=0.99\textwidth]{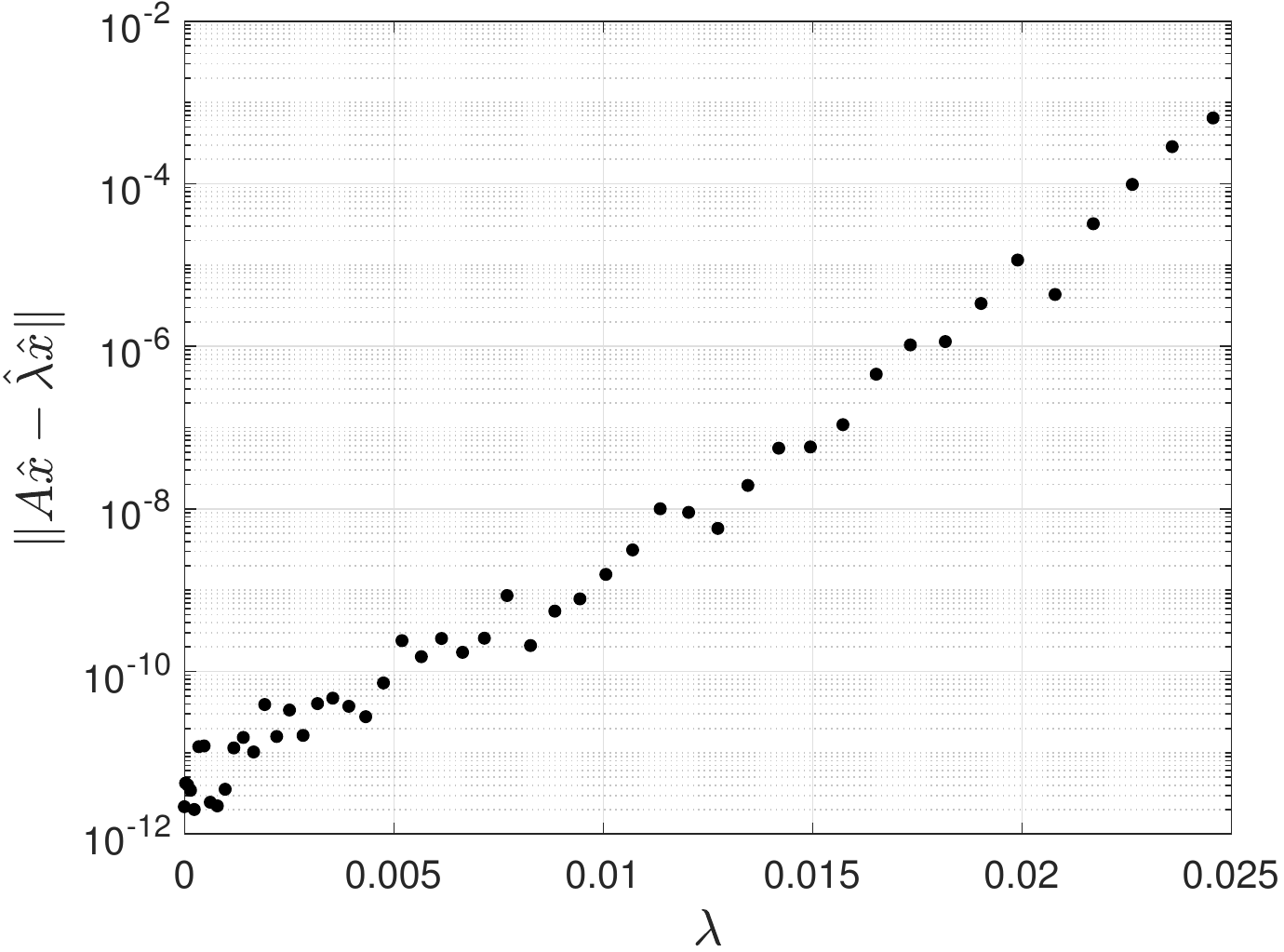}      
  \end{minipage}   
  \caption{Left: convergence of $\sin\angle(\xh_1,x_1)$ 
 (shown as exact), and its bounds~\eqref{eq:sin22}, \eqref{eq:sin2indiv}
and the Davis-Kahan bound~\eqref{eq:classical}. 
Right: scatterplot of $\lambda$ vs. residuals $\|r_i\|=\|A\xh_i-\lh_i\xh_i\|$ 
for $i=1,2,\ldots,k=50$, 
after 20 LOBPCG iterations. Note how $\|r_i\|$ are graded $\|r_i\|\ll \|r_j\|$ for $i\ll j$. 
}
  \label{fig:lob}
\end{figure}

We make several observations. 
First, \eqref{eq:sin2indiv} gave sharp bounds for 
$\sin\angle(\xh_1,x_1)$ when applicable. 
For example after eight LOBPCG iterations, Davis-Kahan's $\sin\theta$ theorem gives bounds $>1$, suggesting $\xh_1$ may have no accuracy at all. Nonetheless, \eqref{eq:sin2indiv} correctly shows that it has at least accuracy $\lesssim 10^{-3}$. 
Second, the bound \eqref{eq:boundvec} is poor throughout, because it takes the entire residual matrix norm $\|R\|$ in the numerator, without respecting the fact that the residuals $\|r_i\|=\|A\xh_i-\lh_i\xh_i\|$ are typically graded and hence $\|r_1\|\ll \|R\|$, as illustrated in Figure~\ref{fig:lob} (right). 
Finally, the asymptotic behavior of the bounds as $\|R\|\rightarrow 0$ (many LOBPCG steps) are also in stark contrast. This is because up to first order in $\|R\|$, \eqref{eq:sin22} and \eqref{eq:sin2indiv} are  $\frac{\|r_1\|}{\Gap}$, whereas Davis-Kahan involves the smaller gap $\frac{\|r_1\|}{\gap}$. 

\subsection{Structured condition number}\label{sec:eigpert}
Here we interpret Theorem~\ref{thm:eigrefine} from the standpoint of perturbation theory. Namely, 
we regard $R$ in~\eqref{eq:At} as a perturbation to the block diagonal matrix 
$\At_0:=\mbox{diag}(\lh_1,\ldots,\lh_k,A_3)$
having eigenvectors $e_1,\ldots,e_k$ (besides others), the first $k$ canonical vectors. Examining $\angle(x_i,\xh_i)$ is equivalent to examining how much the eigenvector $e_i$ of $\At_0$ gets perturbed by $R$. 

In the opening we mentioned that the condition number of an eigenvector is $1/\gap_i$. That is, there exists a perturbation $E$ such that $\At_0+E$ has an eigenvector $\widehat e_i$ with 
\begin{equation}  \label{eq:eihatei}
\sin\angle(e_i,\widehat e_i)=\frac{\|E\|}{\gap_i} + O(\|E\|^2).   
\end{equation}
 Yet, the two bounds in Theorem~\ref{thm:eigrefine} show that writing $R:=\At-\At_0$ (slightly and harmlessly abusing notation), we have 
 \begin{equation}   \label{eq:eihateiGap}
\sin\angle(e_i,\widehat e_i)=\frac{\|r_i\|}{\Gap_i} + O(\|R\|^2)   . 
 \end{equation} 
Note the two changes, both potentially significant: first, $\gap$ is replaced by $\Gap$. Second, the norm of the entire perturbation $\|E\|$ is replaced by the individual $\|r_i\|$, the perturbation only in the $i$th column of $\At_0$. 

An explanation of this effect can be made via structured perturbation analysis. In R-R, the perturbation $R$ in \eqref{eq:At} is highly structured in two ways: the nonzero pattern, and the grading of $\|r_i\|$. For example, the perturbation $E$ that would perturb the eigenvector $e_1$ the most is the $(1,2)$ and $(2,1)$ elements in $\At_0$, as they connect the eigenvalues $\lambda_1$ and $\lambda_2$, resulting in the (unstructured)  condition number $1/|\lambda_1-\lambda_2|=1/\gap_1$. 
However, these are forced to be zero by the R-R construction. 
Within the structured perturbation allowed in R-R, $e_1$ is perturbed most by the $(k+1,1)$ and $(1,k+1)$ elements, assuming for the moment $A_3$ is diagonalized. These elements connect the eigenvalues $\lambda_1$ and $\min(\lambda(A_3))\approx \lambda_{k+1}$, resulting in the structured condition number $1/\Gap_1$. 
Regarding the grading of $\|r_i\|$, the $r_j$ ($j\neq i$) terms have no effect on $\widehat e_i$ up to $O(\|r_j\|^2)$, making $r_i$ the only term that affects the leading term in~\eqref{eq:eihateiGap}. 

\section{Bounds for invariant subspaces}\label{sec:RRangles}
We now turn to bounding errors for invariant subspaces spanned by more than one eigenvector. Besides being the natural object in many applications, it is sometimes necessary to resort to subspaces instead of individual eigenvectors, when multiple or near-multiple eigenvalues are present. For example, if $\gap=O(u)$, none of the above bounds would be useful, as the $O(\frac{u}{\gap})$ term in~\eqref{eq:finalbound} due to roundoff errors is always present. Below we derive bounds that give useful information in such cases. 

We briefly recall the definition of angles between subspaces. The angles $\{\theta_i\}_{i=1}^{k_1}$ between two subspaces spanned by $X\in\mathbb{C}^{n\times k_1},Y\in\mathbb{C}^{n\times k_1}$ 
with orthonormal columns are defined by $\theta_i=\mbox{acos}(\sigma_i(X^*Y))$ and denoted by $\angle(X,Y)$; they are known as the canonical angles or principal angles~\cite[Thm.~6.4.3]{golubbook4th}. Equivalently, we have $\sin\theta_i=\sigma_i(X_\perp^*Y)$ (as can be verified e.g. via the CS decomposition~\cite[Thm.~2.5.2]{golubbook4th}), which is what we use below (and used above for $k_1=1$ to obtain~\eqref{eq:sinx1xh1}). 

To clarify the situation, rewrite~\eqref{eq:At} as
\begin{equation}  \label{reftantheta}
\widetilde A:=
[\Xh_{1}\ \Xh_{2}\ \Xh_3]^\ast A[\Xh_{1}\ \Xh_{2}\ \Xh_3]=
\begin{bmatrix}
\Lh_1&0& R^\ast _1\\
0&\Lh_2& R^\ast _2\\
 R_1& R_2& A_3
\end{bmatrix}, 
\end{equation}
where $[\Xh_{1}\ \Xh_{2}\ \Xh_3]$ is an orthogonal matrix, with 
$[\Xh_{1}\ \Xh_{2}] = Q\Omega$, 
$\Xh_{1}\in\mathbb{C}^{n\times k_1}, \Xh_{2}\in\mathbb{C}^{n\times (k-k_1)}, $ and $\Xh_3\in\mathbb{C}^{n\times (n-k)}$. Our goal is to bound $\unii{\sin\angle(\Xh_1,X_1)}$ from above, where $X_1\in\mathbb{C}^{n\times k_1}$ is a matrix of $k_1$ exact eigenvectors of $A$, i.e., $AX_1=X_1\Lambda_1$. 
Defining $\Xt_1=[\Xh_{1}\ \Xh_{2}\ \Xh_3]^*X_1$, 
we have $\At \Xt_1=\Xt_1\Lambda_1$, so the columns of $\Xt_1$ are eigenvectors of $\At$. 
With the partitioning $\Xt_1 =
\begin{bmatrix}
W\\
Y\\
Z
\end{bmatrix}$ with $W\in\mathbb{C}^{k_1\times k_1}, 
Y\in\mathbb{C}^{(k-k_1)\times k_1}, 
Z\in\mathbb{C}^{(n-k)\times k_1}$, 
it therefore follows that 
$\uniinv{\sin\angle(\Xh_1,X_1)}=
\uniinv{[\Xh_2\ \Xh_3]^*X_1}=
\uniinv{
\begin{bmatrix}
Y\\
Z
\end{bmatrix}}$. This extends~\eqref{eq:sinx1xh1}, and is a key identity in the forthcoming analysis.

Sometimes we deal with the angles between subspaces of different dimensions, say $[\Xh_1\ \Xh_2]\in\mathbb{C}^{n\times k}$ and $X_1\in\mathbb{C}^{n\times k_1}$ with $k_1\leq k$. In this case the angles are defined via $\sin\theta_i=\sigma_i(X_1^*([\Xh_1\ \Xh_2]_\perp))$ for $i=1,\ldots,k_1$. 

Here is the extension of the previous bounds to invariant subspaces. 
Note that $\gap$ and $\Gap$ are redefined; we use the same notation as they reduce to the same values when $k_1=1$. 

\begin{theorem}\label{thm:mainsubspace}
Let $A,\At$ be as in~\eqref{reftantheta}, 
with $(\Lh_{1},\Xh_1)$ being $k_1$ Ritz pairs. 
Let $(\Lambda_1,X_1)$ be a set of $k_1$ exact eigenpairs $AX_1=X_1\Lambda_1$. 
Let $\Gap=\min|\lambda(\Lambda_1)-\lambda(A_3)|$ and 
$\gap=\min|\lambda(\Lambda_1)-\lambda(\Lh_2)|$. 
Then writing 
$R=[R_1\ R_2]:=[R_1\ r_{k_1+1},\ldots,r_{k}]\in\mathbb{C}^{n\times k}$
 where $R_1\in\mathbb{C}^{n\times k_1}$, we have 
\begin{equation}  \label{eq:boundvecsubspace}
\uniinv{\sin\angle(X,\Xh)}
\leq  \frac{\uniinv{R}}{\Gap}(1+\frac{\|R_2\|}{\gap}), 
\qquad
\|\sin\angle(X,\Xh)\|_{2,F}\leq 
\frac{\|R\|_{2,F}}{\Gap}\sqrt{1+\frac{\|R_2\|^2}{\gap^2}}. 
\end{equation}
Moreover, if $\Gap>\frac{\|R_2\|^2}{\gap}$ then 
\begin{align}
\left\|\sin\angle(X,\Xh)\right\|_{2,F}   \label{eq:sin22subF}&\leq \frac{\|R_1\|_{2,F}}{\Gap-\frac{\|R_2\|^2}{\gap}}\sqrt{1+\frac{\|R_2\|^2}{\gap^2}},     
\end{align}
and if $\Gap>\sum_{i=k_1+1}^{k}\frac{\|r_{i}\|^2}{\min|\lambda(\Lambda_1)  -\lh_{i}|}
$ then 
\begin{equation}
  \label{eq:sin22subindF}
\left\|\sin\angle(X,\Xh)\right\|_{2,F}
\leq \frac{\|R_1\|_{2,F}}{
\Gap-\sum_{i=k_1+1}^{k}\frac{\|r_{i}\|_{2}^2}{\min|\lambda(\Lambda_1)  -\lh_{i}|}
}\sqrt{1+
\left(\sum_{i=k_1+1}^{k}\frac{\|r_{i}\|}{\min|\lambda(\Lambda_1) -\lh_{i}|}\right)^2}.
\end{equation}
\end{theorem}
\begin{proof}
The proof mimics that of Theorem~\ref{thm:eigrefine}, extending the discussion from vectors to subspaces. 
Let 
$\Xt_1 = 
[\Xh_1,\Xh_2,\Xh_3]X_1=
\begin{bmatrix}
W\\  Y\\Z
\end{bmatrix}\in\mathbb{C}^{n\times k_1}$ be 
an invariant subspace of $\At$  such that 
\begin{equation}
  \label{eq:Atblock}
\At
\begin{bmatrix}
 W\\ Y\\Z
\end{bmatrix}=
\begin{bmatrix}
 W\\ Y\\Z
\end{bmatrix}\Lambda_1.   
\end{equation}

Then the bottom part of the equation gives 
\begin{equation}
  \label{eq:Zeqn}
Z\Lambda_1 -A_3 Z =[R_1\ R_2]
\begin{bmatrix}
  W\\Y
\end{bmatrix}=R\begin{bmatrix}
  W\\Y
\end{bmatrix}.    
\end{equation}
Using a well-known bound for Sylvester's equations~(e.g.~\cite[Lem.~2]{li2015convergence}, \cite[Ch.~V]{stewart-sun:1990}), 
along with the fact $\min(\lambda(\Lambda_1)-\lambda(A_3))=\Gap$, 
we obtain 
\begin{equation}  \label{eq:dkspecialsub}
\uniinv{Z}
\leq \frac{\uniinv{R}\left\|\begin{bmatrix}  W\\Y\end{bmatrix}\right\|}{\Gap  }
\leq \frac{\uniinv{R}}{\Gap}, 
\end{equation}
where for the last inequality we used the fact $\uniinv{XY}\leq \uniinv{X}\|Y\|$~\cite[Cor. 3.5.10]{hornjohntopics}. 
As in~\eqref{eq:dkspecial}, this is the generalized Davis-Kahan $\sin\theta$ theorem. 

From the second block of~\eqref{eq:Atblock} we have 
\begin{equation}\label{eq:Ylam}
Y\Lambda_1 -\Lh_{2}Y = R_2^*Z,  
\end{equation}
hence again from the Sylvester equation bound
\begin{equation}
  \label{eq:Ynrmuniinv}
\uniinv{Y}
\leq \frac{\|R_2\|\uniinv{Z}}{\gap}.  
\end{equation}
Together with~\eqref{eq:dkspecialsub} we obtain 
\begin{equation}
  \label{eq:y2knyasub}
\uniinv{Y}
\leq \frac{\uniinv{R}\|R_2\|}{\gap\cdot \Gap} . 
\end{equation}
Therefore, we conclude that 
\[
\uniinv{\sin\angle(X,\Xh)}= 
\uniinv{
\begin{bmatrix}
Y\\Z
\end{bmatrix}}
 \leq 
\frac{\uniinv{R}}{\Gap}(1+\frac{\|R_2\|}{\gap}), 
\]
the first inequality in~\eqref{eq:boundvecsubspace}. 
For the spectral and Frobenius norms, using the 
 stronger inequality 
 \begin{equation}   \label{eq:2Fget}
\left\|
  \begin{bmatrix}
    A\\B
  \end{bmatrix}\right\|_{2,F}\leq 
\sqrt{\left\|A \right\|_{2,F}^2
+\left\|B\right\|_{2,F}^2
}   ,
 \end{equation}
we obtain the second result in~\eqref{eq:boundvecsubspace}. 

We next prove~\eqref{eq:sin22subF}. 
\ignore{
We first improve the bound~\eqref{eq:dkspecialsub}. 
As before we have 
\begin{equation}  \label{eq:Ylam}
Y\Lambda_1 -\Lh_{2}Y = R_2^*Z,  
\end{equation}
so $\|Y\|\leq \frac{\|R_2\|\|Z\|}{\gap}$ by the bound for Sylvester's equation. Now 
\[
Z\Lambda_1 -A_3 Z =[R_1\ R_2]
\begin{bmatrix}
  W\\Y
\end{bmatrix}, 
\]
} 
From~\eqref{eq:Zeqn} we obtain $\uniinv{Z\Lambda_1 -A_3 Z } -\uniinv{R_2Y}\leq \uniinv{R_1W}$, hence using~\eqref{eq:Ynrmuniinv} we have 
\[
\uniinv{Z\Lambda_1 -A_3 Z }-
\frac{\|R_2\|^2\uniinv{Z}}{\gap}
\leq \uniinv{R_1W}. 
\]
Again using the Sylvester equation bound $\uniinv{Z\Lambda_1 -A_3 Z}\geq \Gap \uniinv{Z}$, 
we therefore obtain 
\[
(\Gap-\frac{\|R_2\|^2}{\gap})\uniinv{Z} \leq \uniinv{R_1W}. 
\]
Hence using the assumption $\Gap>\frac{\|R_2\|^2}{\gap}$ and the trivial bound $\|W\|\leq 1$  along with $\uniinv{R_1W}\leq \uniinv{R_1}\|W\|$, we obtain
\[
\uniinv{Z} \leq \frac{\uniinv{R_1}}{\Gap-\frac{\|R_2\|^2}{\gap}}. 
\]
Finally, using~\eqref{eq:2Fget} again we obtain
\[
\left\|\sin\angle(X,\Xh )\right\|_{2,F} = 
\left\|
\begin{bmatrix}
Y\\Z
\end{bmatrix}
\right\|_{2,F}
\leq \frac{\|R_1\|_{2,F}}{\Gap-\frac{\|R_2\|^2}{\gap}}\sqrt{1+\frac{\|R_2\|^2}{\gap^2}}, 
\]
giving~\eqref{eq:sin22subF}. 

It remains to establish 
\eqref{eq:sin22subindF}. 
Taking the $i$th row of~\eqref{eq:Ylam} gives 
$y_{i}\Lambda_1  -\lh_{k_1+i}y_{i} = r_{k_1+i}^*Z$ for $i=1,\ldots,k-k_1$, where $y_{i}$ is the $i$th row of $Y$. 
Hence 
\begin{equation}  \label{eq:ithysub}
\uniinv{y_{i}}\leq 
 \frac{\|r_{k_1+i}\|\uniinv{Z}}{\min|\lambda(\Lambda_1)  -\lh_{k_1+i}|}, \quad i=1,\ldots,k-k_1. 
\end{equation}
We also have 
$
Z\Lambda_1 -A_3 Z =[R_1,r_{k_1+1},\ldots,r_k]
\begin{bmatrix}
  W\\y_1\\\vdots\\y_{k-k_1}
\end{bmatrix}. 
$
This gives $(Z\Lambda_1 -A_3 Z)-[r_{k_1+1},\ldots,r_k]\begin{bmatrix}
  y_1\\\vdots\\y_{k-k_1}
\end{bmatrix}=R_1W$, 
so 
using~\eqref{eq:ithysub}
we obtain 
\begin{equation}
  \label{eq:rindivsubspace}
\uniinv{Z\Lambda_1 -A_3 Z} -
\sum_{i=k_1+1}^{k}\frac{\|r_{i}\|^2\uniinv{Z}}{
\min|\lambda(\Lambda_1)  -\lh_{i}|}
\leq \uniinv{R_1W}  . 
\end{equation}
Since $\uniinv{Z\Lambda_1 -A_3 Z }\geq \uniinv{Z}/\Gap$ as before, this gives
\[
(\Gap-
\sum_{i=k_1+1}^{k}\frac{\|r_{i}\|^2}{\min|\lambda(\Lambda_1)  -\lh_{i}|}
)\uniinv{Z} \leq \uniinv{R_1W}\leq \uniinv{R_1}\|W\|. 
\]
Hence using the assumption $\Gap>
\sum_{i=k_1+1}^{k}\frac{\|r_{i}\|^2}{\min|\lambda(\Lambda_1)  -\lh_{i}|}$ and the trivial bound $\|W\|\leq 1$ we obtain
\[
\uniinv{Z} \leq \frac{\uniinv{R_1}}{
\Gap-\sum_{i=k_1+1}^{k}\frac{\|r_{i}\|^2}{\min|\lambda(\Lambda_1)  -\lh_{i}|}
}.
\]
We use the fact $\uniinv{\sin\angle(X,\Xh)} = 
\uniinv{
\begin{bmatrix}
Y\\Z
\end{bmatrix}}
$ together with~\eqref{eq:ithysub} to complete the proof of 
\eqref{eq:sin22subindF}, again using~\eqref{eq:2Fget}. 
\end{proof}

\ignore{
Then the bottom part of the equation gives 
\[
 Z\Lambda - FZ =RY. 
\]
Using 
a well-known bound for Sylvester's equation  we obtain 
\begin{equation}  \label{eq:dkspecial}
\|Z\|\leq  \|R\|/\Gap  . 
\end{equation}
Note that the denominator is $\Gap$, not $\gap$. Note that the last bound is Davis-Kahan when the perturbation is off-diagonal. 

Now, write $Y=\begin{bmatrix}  Y_1\\Y_2\end{bmatrix}$ where $Y_1$ contains the desired part. Then from the second part of the first block of 
$A
\begin{bmatrix}
  Y\\Z
\end{bmatrix}=\lambda 
\begin{bmatrix}
  Y\\Z
\end{bmatrix}$ we obtain (writing $A_1 =
\begin{bmatrix}
  A_{11}&   0\\
  0&   A_{22}\end{bmatrix}$; crucially, the off-diagonal blocks are zero due to Rayleigh-Ritz)
\[
Y_2\Lambda - A_{22}Y_2 = R^*Z,
\]
from which we obtain the important bound 
\[
\|Y_2\| \leq \frac{\|R\|}{\gap}\|Z\|.
\]
Combining with~\eqref{eq:dkspecial} we obtain 
\[
\|Y_2\| =\|R\|^2/(\gap\cdot \Gap). 
\]
Therefore, we conclude that 
\[
\sin\angle(X,\xh )^2 = 
\left\|
\begin{bmatrix}
Y_2\\Z
\end{bmatrix}
\right\|^2 \leq  \frac{\|R\|^4}{\gap^2\cdot \Gap^2}+\frac{\|R\|^2}{\Gap^2}, 
\]
}


Four remarks are in order. 
\begin{remark}[Vector vs. subspace bounds]
  The $\|\cdot\|_{2,F}$ bounds in Theorem~\ref{thm:mainsubspace} reduce to the vector bounds in the previous sections by taking $k_1=1$. Thus they can be regarded as proper generalizations. 
\end{remark}
\begin{remark}[Bounds for unitarily invariant norms]
Just like the two bounds in~\eqref{eq:boundvecsubspace}, 
\eqref{eq:sin22subF} and~\eqref{eq:sin22subindF} have their counterparts applicable to any unitarily invariant norm, obtained by replacing the final factors of the form $\sqrt{1+b^2}$ by $1+b$. 
The proof is identical up to the very end, where we use the bound 
$\uniinv{\begin{bmatrix}A\\B\end{bmatrix}}\leq \uniinv{A}+\uniinv{B}$ instead of~\eqref{eq:2Fget}. 
\end{remark}
\begin{remark}[Question 10.3 by Davis-Kahan]
At the end of their landmark paper, Davis and Kahan~\cite{daviskahan} suggest four open problems. Among them, Question 10.2 asks for an extension of their theorems to the case where $\mathbb{C}^n$ is split into a pair of three (instead of two) subspaces  in two ways, $X_1,X_2,X_3$ (exact eigenspaces) and $\Xh_1,\Xh_2,\Xh_3$ (approximate ones). 
Namely, using information such as the Ritz values and residuals, can one bound the subspace angles? 
We argue that the above results give an answer---the setting in~\eqref{reftantheta} is precisely in this form, and Theorem~\ref{eq:boundvecsubspace} gives sharp bounds for $\unii{\sin\angle(X_1,\Xh_1)}$. 
\end{remark}
\begin{remark}[On the definition of \gap]\label{rem:gapdef}
As mentioned in footnote~\ref{foot1}, there is a difference in what is measured between our $\gap,\Gap$ in~\eqref{eq:Gapgapdef} and $\ggap$ in the classical treatments~\cite{daviskahan},\cite[Ch.~11]{parlettsym}. 
For example,~\eqref{eq:boundvecsubspace} reduces to $\frac{\|r_1\|}{\Gap}$ when $k_1=k=1$ (hence $R_2$ is empty), but $\Gap$ here 
 is the difference between $\lambda$ (an exact desired eigenvalue) and $\lambda(A_3)$ (approximations to undesired eigenvalues),
and the same can be said of $\gap$ in~\eqref{eq:Gapgapdef}. 
By contrast, $\ggap$\ in \eqref{eq:classical} is the distance between the approximate desired eigenvalue $\lh$ and exact undesired eigenvalues. 
It turns out that for~\eqref{eq:classical}, both gaps are applicable---indeed, we can obtain~\eqref{eq:classical} from Theorem~\ref{thm:mainsubspace}, takinig $k_1=k=n-1$: note that 
$\|\sin\angle(x,\xh)\|=
\|\xh^TX_\perp\|= \|\sin\min\angle(X_\perp,\Xh_\perp)\|$
(this can be verified e.g. via the CS decomposition~\cite[Thm.~2.5.2]{golubbook4th}; here $[x,X_\perp]$ and $[\xh,\Xh_\perp]$ are orthogonal), and 
invoke~\eqref{eq:boundvecsubspace} taking  
$X\leftarrow X_\perp$, $\Xh\leftarrow \Xh_\perp$, and $\Lambda_2$ empty. Then $\Gap$ in \eqref{eq:boundvecsubspace} becomes $\gap$ in \eqref{eq:classical}, and since the residuals are related by $R=r_1^T$ their norms are the same $\|R\|=\|r_1\|$, so
\eqref{eq:boundvecsubspace} reduces precisely to~\eqref{eq:classical}. In other words, the bound~\eqref{eq:classical} holds regardless of which definition of \gap\ is used. However, the analysis in this paper uses $\gap,\Gap$ in \eqref{eq:Gapgapdef}, and we have not proved that our theorems hold with $\ggap$.
\end{remark}
\begin{remark}[Proof techniques]
The reader might have also noticed that the proofs above are basically repeated applications of well-known norm inequalities in matrix analysis. One might then wonder, why do they appear to give stronger results than previous ones? The answer appears to lie in~\eqref{eq:At}---the simple but crucial unitary transformation from $A$ to $\At$ that simplifies the task to bounding $\|Y\|,\|Z\|$, as in \eqref{eq:sinx1xh1}. By contrast, most classical results work with $A$ and start from the residual equation $A\Xh_1-\Xh_1\Lh_1=R$ and 
derive bounds on $\angle(X_1,\Xh_1)$: for example the Davis-Kahan $\sin\theta$ theorem can be obtained essentially by left-multiplying $X_3^T$, taking $X_2$ empty. Knyazev~\cite[Sec.~4]{knyazevmc97} employed ingenious techniques to obtain (among others) essentially~\eqref{eq:y2knyasub}. Obtaining ``sharper'' bounds like~\eqref{eq:rindivsubspace} in a similar manner appears to be highly challenging. Once we reformulate the problem as in~\eqref{eq:At}--\eqref{eq:sinx1xh1}, the derivation becomes significantly simpler (in the author's opinion). 
\end{remark}


\section{SVD}
We now present an SVD analogue of Theorem~\ref{thm:mainsubspace}, 
 deriving bounds for the accuracy of singular vectors and singular subspaces obtained by a Petrov-Galerkin projection method. 
Such  methods proceed as follows: 
project $A$ onto lower-dimensional trial subspaces spanned by $\widehat U\in\mathbb{C}^{m\times k_m}, \widehat V\in\mathbb{C}^{n\times k_n}$ having orthonormal columns (for how to choose $\Uh,\Vh$ see e.g.~\cite{baitemplates,halko2011finding,wu2017primme_svds}), 
compute the SVD of the small $k_m\times k_n$ matrix 
$\Uh^*A\Vh = \Ut\Sh \Vt^*$
 and obtain an
approximate economical SVD as 
$A\approx (\Uh \Ut)\Sh (\Vh\Vt)^*$, which is 
of rank $\leq \min(k_m,k_n)$. 
 Some of the columns of $\Uh \Ut$ and $\Vh \Vt$ 
then approximate the exact left and right singular vectors of $A$. Our goal is to quantify their accuracy. 
We focus on the most frequently encountered case where an approximate SVD is sought, that is, the leading singular vectors are being approximated.




\begin{theorem}\label{thm:svdmain}
Let $A\in\mathbb{C}^{m\times n}$ with $m\geq n$, of the form
\begin{equation}  \label{eq:svdbasis}
[\widehat U_1\ \widehat U_2\ \widehat U_3]^{\ast}A[\widehat V_1\ \widehat V_2\ \widehat V_3]=
\begin{bmatrix}
\Sh_1&0& R_1\\
0&\Sh_2& R_2\\
 S_1& S_2& A_3
\end{bmatrix}=:\At, 
\end{equation}
where 
$[\widehat U_1\ \widehat U_2\ \widehat U_3]$ and 
$[\widehat V_1\ \widehat V_2\ \widehat V_3]$ are square unitary, and 
$\Sh_1\in\mathbb{R}^{k_1\times k_1}$, $\Sh_2\in\mathbb{R}^{(k_m-k_1)\times (k_n-k_1)}$ with 
$\big[\begin{smallmatrix}\Sh_1 & \\& \Sh_2  \end{smallmatrix} 
\big]$ equal to $\big[\begin{smallmatrix}{\rm diag}(\sh_1,\sh_2,\ldots,\sh_k)  \\ 0\end{smallmatrix} \big]$
if $k_m\geq k_n=k$, and 
$\big[\begin{smallmatrix}{\rm diag}(\sh_1,\sh_2,\ldots,\sh_k)\ 0_{k\times (k_n-k)}\end{smallmatrix} \big]$ 
if $k=k_m< k_n$. 
Let $(\Sigma_1,U_1,V_1)$ be the set of $k_1$ leading singular triplets of $A$. 
Define $\Gap=\min(\sigma(\Sh_1)-\sigma(A_3))$ and 
$\gap = \sigma_{\min}(\Sigma_1)-\|\Sh_2\|$, and suppose that $\Gap,\gap>0$. Write $[S_1\ S_2]=S$, $\big[
\begin{smallmatrix}
  R_1\\R_2
\end{smallmatrix}
\big]=R$ and for brevity define
$\uni{\Theta}:=\max(\unii{\sin\angle(U_1,\widehat  U_1)},\unii{\sin\angle(V_1,\widehat  V_1)})$. Then we have 
\begin{equation}  \label{eq:boundvecsvd}
\uni{\Theta}\leq 
\frac{\max(\uni{R},\uni{S})}{\Gap}\left(1+\frac{\max(\|R_2\|,\|S_2\|)}{\gap}\right).
\end{equation}
Moreover, provided that 
$\Gap>\frac{\max(\|S_2\|,\|R_2\|)^2}{\gap}$, we have 
\begin{align}
\uni{\Theta}\leq 
\label{eq:sin22subSVD}
\frac{\max(\uni{S_1},\uni{R_1})}{\Gap-\frac{\max(\|S_2\|,\|R_2\|)^2}{\gap}}
 \left(1+\frac{\max(\|R_2\|,\| S_2\|) }{\gap}\right). 
\end{align}
Finally, define $k':=\max(k_m-k_1,k_n-k_1)$, 
and 
denote by $r_{2i}^T$ the $i$th row of $R_{21}$ and 
by  $s_{2i}$ the $i$th column of $S_{21}$ (setting $r_{2i}=0$ for $i>k_m-k_1$ and $s_{2i}=0$ for $i>k_n-k_1$, and $\sh_{k_1+i}=0$ for $i>\min(k_m,k_n)-k_1$). 
If $\Gap>\sum_{i=1}^{k'}
\frac{ \max(\|r_{2i}\|,\|s_{2i}\|)^2}{\sigma_{\min}(\Sigma_1)-\sh_{k_1+i}}$, then 
\begin{equation}
  \label{eq:sin22subindSVD}
\uni{\Theta}\leq 
\frac{
\max(\uni{S_1},\uni{R_1})}{\Gap-\sum_{i=1}^{k'}
\frac{ \max(\|r_{2i}\|,\|s_{2i}\|)^2}{\sigma_{\min}(\Sigma_1)-\sh_{k_1+i}}}
\left(1+\sum_{i=1}^{k'}
\frac{ \max(\|r_{2i}\|,\|s_{2i}\|)}{\sigma_{\min}(\Sigma_1)-\sh_{k_1+i}}
\right)
.
%
\end{equation}
\end{theorem}
Though not displayed for brevity, slightly improved bounds for $\|\cdot\|_{2,F}$ analogous to those in Theorem~\ref{thm:mainsubspace} are available for each bound above. 
The derivation is again the same, using the inequality $\left\|\big[\begin{smallmatrix}A\\B  \end{smallmatrix}\big]\right\|_{2,F}\leq \sqrt{\left\|A \right\|_{2,F}^2+\left\|B\right\|_{2,F}^2}$.  

\begin{proof}

Let $\left(\Sigma_1,\Ut_1,\Vt_1\right)$ be a set of exact singular triplets 
of $\tilde A$, 
i.e., $\tilde A\Vt_1=\Ut_1\Sigma_1$ and 
$\Ut_1^*\tilde A=\Sigma_1\Vt_1^*$. 
Write $\Vt_1=\begin{bmatrix}\Vt_{11}\\\Vt_{21}\\\Vt_{31}  \end{bmatrix}, \Ut_1=\begin{bmatrix}\Ut_{11}  \\\Ut_{21}\\\Ut_{31}  \end{bmatrix}$, so that 
\begin{equation}
  \label{eq:Av1}
 \begin{bmatrix}
\Sh_1&0& R_1\\
0&\Sh_2& R_2\\
 S_1& S_2& A_3
\end{bmatrix}
\begin{bmatrix}\Vt_{11}\\\Vt_{21}\\\Vt_{31}  \end{bmatrix}=
\begin{bmatrix}\Ut_{11}  \\\Ut_{21}\\\Ut_{31}  \end{bmatrix}\Sigma_1,
\end{equation}
and
\begin{equation}
  \label{eq:Au}
\begin{bmatrix}\Ut_{11}^{\ast} \ \Ut_{21}^{\ast}\ \Ut_{31}^{\ast}  \end{bmatrix}
\begin{bmatrix}
\Sh_1&0& R_1\\
0&\Sh_2& R_2\\
 S_1& S_2& A_3
\end{bmatrix}=
\Sigma_1\begin{bmatrix}\Vt_{11}^{\ast}\ \Vt_{21}^{\ast}\ \Vt_{31}^{\ast}  \end{bmatrix}  .
\end{equation}

As in the previous sections, we have the crucial identities 
\begin{equation}  \label{eq:u1andv1}
\uni{\sin\angle(U_1,\widehat U_1)}=\uni{  \begin{bmatrix}\Ut_{21}    \\\Ut_{31}    \end{bmatrix}}, \quad
\uni{\sin\angle(V_1,\widehat V_1)}=\uni{  \begin{bmatrix}\Vt_{21}    \\\Vt_{31}    \end{bmatrix}}.
\end{equation} 
To prove the theorem we first  bound
$\unii{\Ut_{21}}$  with respect to $\unii{\Ut_{31}}$, and similarly 
bound $\unii{\Vt_{21}}$  with respect to $\unii{\Vt_{31}}$. 

From the second block of~\eqref{eq:Av1} 
 we obtain 
\begin{equation}  \label{eq:forsvd1}
\Sh_2\Vt_{21}+ R_2\Vt_{31}=\Ut_{21}\Sigma_1,  
\end{equation}
and the second block of \eqref{eq:Au} gives 
\begin{equation}  \label{eq:forsvd2}
\Ut_{21}^{\ast}\Sh_2+\Ut_{31}^{\ast} S_2=\Sigma_1\Vt_{21}^{\ast}.  
\end{equation}
Taking norms and using the triangular inequality 
and the fact $\sigma_{\min}(X)\uniinv{Y}\leq \uniinv{XY}\leq \|X\|\uniinv{Y}$ (the lower bound holds if $X\in\mathbb{C}^{m\times n}, m\geq n$) in \eqref{eq:forsvd1} and \eqref{eq:forsvd2}, we obtain
 \begin{equation}
   \label{eq:del1}
   \begin{split}     
\uniinv{\Ut_{21}}\sigma_{\min}(\Sigma_1)-\uniinv{\Vt_{21}}\|\Sh_2\|&\leq \uniinv{ R_2\Vt_{31}},   \\
\uniinv{\Vt_{21}}\sigma_{\min}(\Sigma_1)-\uniinv{\Ut_{21}}\|\Sh_2\|&\leq \uniinv{\Ut_{31}^{\ast} S_2}. 
   \end{split}
 \end{equation}
By adding the first inequality times $\sigma_{\min}(\Sigma_1)$ and the second inequality times $\|\Sh_2\|$, we eliminate the $\unii{\Vt_{21}}$ term, 
and recalling the assumption 
$\sigma_{\min}(\Sigma_1)>\|\Sh_2\|$ we obtain 
\begin{equation}  \nonumber 
\uni{\Ut_{21}}\leq \frac{ \sigma_{\min}(\Sigma_1)\unii{R_2\Vt_{31}}+ \|\Sh_2\|\unii{\Ut_{31}^{\ast}S_2}}{(\sigma_{\min}(\Sigma_1))^2-\|\Sh_2\|^2}.
\end{equation}
Eliminating $\unii{\Ut_{21}}$ from~\eqref{eq:del1} 
similarly yields 
\begin{equation} \nonumber 
\uni{\Vt_{21}}\leq \frac{ \sigma_{\min}(\Sigma_1)\unii{\Ut_{31}^{\ast}S_2}+ \|\Sh_2\|\unii{R_2\Vt_{31}}}{(\sigma_{\min}(\Sigma_1))^2-\|\Sh_2\|^2}.
\end{equation}
Combining these two inequalities we obtain 
\begin{align}
\max(\uni{\Ut_{21}},\uni{\Vt_{21}})&\leq \frac{\max(\unii{\Ut_{31}^{\ast}S_2},\unii{R_2\Vt_{31}}) }{\sigma_{\min}(\Sigma_1)-\|\Sh_2\|}\nonumber\\
&\leq \frac{\max(\unii{\Ut_{31}},\unii{\Vt_{31}}) \max(\|R_2\|,\|S_2\|) }{\sigma_{\min}(\Sigma_1)-\|\Sh_2\|}\nonumber\\
&=  \frac{\max(\|R_2\|,\| S_2\|) }{\gap}
\max(\uni{\Ut_{31}},\uni{\Vt_{31}}).    \label{eq:21to31sig}
\end{align}
Together with \eqref{eq:u1andv1} it follows that
\begin{align}
\max&(\uni{\sin\angle(U_1,\widehat  U_1)},\uni{\sin\angle(V_1,\widehat  V_1)})\nonumber=\max\left(
\uni{  \begin{bmatrix}\Ut_{21}    \\\Ut_{31}    \end{bmatrix}},
\uni{  \begin{bmatrix}\Vt_{21}    \\\Vt_{31}    \end{bmatrix}}\right)\nonumber\\
&\leq\max(
\unii{  \Ut_{21}}  +\unii{  \Ut_{31}},
\unii{  \Vt_{21}}  +\unii{  \Vt_{31}})
\nonumber\\
&\leq (1+\frac{\max(\|R_2\|,\| S_2\|) }{\gap})
\max(\uni{  \Ut_{31}},\uni{  \Vt_{31}}).\label{eq:sinuuvv}
\end{align}
The remaining task is to bound $\max(\unii{  \Ut_{31}},\unii{  \Vt_{31}})$. 
The bottom block of~\eqref{eq:Av1} gives 
\begin{equation}  \label{eq:Ut31S1S2etc}
 S_1\Vt_{11}+ S_2\Vt_{21}+A_3\Vt_{31}=\Ut_{31}  \Sigma_1.   
\end{equation}
Hence recalling that $[S_1\ S_2]=S$ we have 
\begin{equation}
  \label{eq:simple1}
\sigma_{\min}(\Sigma_1)\Uniinv{\Ut_{31}  }\leq  \Uniinv{S}+\|A_3\|\Uniinv{\Vt_{31}}.   
\end{equation}

Similarly, from the last block of~\eqref{eq:Au} 
\begin{equation}
  \label{eq:eq:Ut31S1S2etc2}
\Ut_{11}^{\ast}R_1+ \Ut_{21}^{\ast}R_2+ \Ut_{31}^{\ast} A_3=\Sigma_1\Vt_{31}^{\ast},    
\end{equation}
we obtain 
\begin{equation}  \label{eq:simple2}
\sigma_{\min}(\Sigma_1)\Uniinv{\Vt_{31}  }\leq  
\Uniinv{ R}
+\|A_3\|\Uniinv{\Ut_{31}}
\end{equation}
We multiply~\eqref{eq:simple1} by $\sigma_{\min}(\Sigma_1)$ and~\eqref{eq:simple2} by $\|A_3\|$, and add them to eliminate the $\unii{\Vt_{31}}$ terms, to obtain 
\begin{align*}
(\sigma_{\min}(\Sigma_1)^2-\|A_3\|^2)\Uniinv{\Ut_{31}}
&\leq 
\|A_3\|\Uniinv{R}
+\sigma_{\min}(\Sigma_1)\Uniinv{S}
\\
&\leq 
(\sigma_{\min}(\Sigma_1)+\|A_3\|)\max(\Uniinv{R},\Uniinv{S}).
\end{align*}
Hence by the assumption $\Gap=\sigma_{\min}(\Sigma_1)-\|A_3\|>0$, we have 
\[
\Uniinv{\Ut_{31}}\leq \frac{\max(\Uniinv{R},\Uniinv{S})}{\Gap}. 
\]
Eliminating the $\unii{\Ut_{31}}$ terms from \eqref{eq:simple1} and \eqref{eq:simple2} yields the same bound for $\unii{\Vt_{31}}$, hence 
\[
\max(\Uniinv{\Ut_{31}},\Uniinv{\Vt_{31}})\leq 
\frac{\max(\Uniinv{R},\Uniinv{S})}{\Gap}. 
\]
Combine this with~\eqref{eq:21to31sig} and~\eqref{eq:u1andv1} to obtain~\eqref{eq:boundvecsvd}. 

We next prove~\eqref{eq:sin22subSVD}. 
From~\eqref{eq:Ut31S1S2etc} we also obtain
\begin{equation}
  \label{eq:Ut311}
\sigma_{\min}(\Sigma_1)\Uniinv{\Ut_{31}}\leq \Uniinv{S_1}+\|S_2\|\Uniinv{\Vt_{21}}+\|A_3\|\Uniinv{\Vt_{31}} ,
\end{equation}
and from~\eqref{eq:eq:Ut31S1S2etc2}, 
\begin{equation}
  \label{eq:Vt311}
\sigma_{\min}(\Sigma_1)\Uniinv{\Vt_{31}}
\leq \Uniinv{R_1}+\Uniinv{\Ut_{21}}\|R_2\|+\Uniinv{\Ut_{31}}\|A_3\|.   
\end{equation}
Again eliminate the $\|\Vt_{31}\|$ terms
by multiplying~\eqref{eq:Ut311} by $\sigma_{\min}(\Sigma_1)$ and~\eqref{eq:Vt311} by $\|A_3\|$, and adding them:
\begin{align*}
(\sigma_{\min}&(\Sigma_1)^2-\|A_3\|^2)\Uniinv{\Ut_{31}}\\
&\leq \sigma_{\min}(\Sigma_1)(\Uniinv{S_1}+\|S_2\|\Uniinv{\Vt_{21}} )
+\|A_3\|(\Uniinv{R_1}+\|R_2\|\Uniinv{\Ut_{21}})  \\
&\leq 
(\sigma_{\min}(\Sigma_1)+\|A_3\|)(\max(\Uniinv{S_1},\Uniinv{R_1})+\max(\|S_2\|,\|R_2\|)
\max(\Uniinv{\Ut_{21}},\Uniinv{\Vt_{21}})). 
\end{align*}
Therefore, using~\eqref{eq:21to31sig} we obtain 
\begin{align}
\uni{\Ut_{31}}
&\leq \frac{ 
\max(\uni{S_1},\uni{R_1})+\max(\|S_2\|,\|R_2\|)
\max(\uni{\Ut_{21}},\uni{\Vt_{21}})}{\sigma_{\min}(\Sigma_1)-\|A_3\|}\nonumber\\
&\leq \frac{1}{\Gap}\nonumber
\left(\max(\uni{S_1},\uni{R_1})+\frac{\max(\|S_2\|,\|R_2\|)^2}{\gap}\max(\uni{\Ut_{31}},\uni{\Vt_{31}})\right).
\end{align}
As before, eliminating $\unii{\Ut_{31}}$ from~\eqref{eq:Ut311} and~\eqref{eq:Vt311} yields the same bound for $\unii{\Vt_{31}}$, hence 
\begin{align*}
\max(\Uniinv{\Ut_{31}},\Uniinv{\Vt_{31}})
&\leq \frac{1}{\Gap}
\left(\max(\Uniinv{S_1},\Uniinv{R_1})+\frac{\max(\|S_2\|,\|R_2\|)^2}{\gap}\max(\Uniinv{\Ut_{31}},\Uniinv{\Vt_{31}})\right). 
\end{align*}
Therefore using the assumption $\Gap>\frac{\max(\|S_2\|,\|R_2\|)^2}{\gap}$ we obtain
\begin{align*}
\max(\Uniinv{\Ut_{31}},\Uniinv{\Vt_{31}})
&\leq \frac{\max(\Uniinv{S_1},\Uniinv{R_1})}{\Gap-\frac{\max(\|S_2\|,\|R_2\|)^2}{\gap}}. 
\end{align*}
Together with~\eqref{eq:sinuuvv} 
we obtain~\eqref{eq:sin22subSVD}.   

The remaining task is to establish~\eqref{eq:sin22subindSVD}. For this, we 
revisit \eqref{eq:forsvd1}, \eqref{eq:forsvd2}, and now bound the individual $\Ut_{21i}$, the $i$th row of $\Ut_{21}$. 
We obtain 
\[
\Uniinv{\Ut_{21i}}\sigma_{\min}(\Sigma_1)-\Uniinv{\Vt_{21i}}\sh_{k_1+i}\leq \Uniinv{ r_{2i}^T\Vt_{31}}, \quad 
\Uniinv{\Vt_{21i}}\sigma_{\min}(\Sigma_1)-\Uniinv{\Ut_{21i}}\sh_{k_1+i}\leq \Uniinv{ \Ut_{31}^*s_{2i}}, 
\]
Eliminating $\Vt_{21i}$ and $\Ut_{21i}$ as before gives 
\begin{equation}  \nonumber 
\Uniinv{\Ut_{21i}}\leq \frac{ \sigma_{\min}(\Sigma_1)\Uniinv{r_{2i}^T\Vt_{31}}+ \sh_{k_1+i}\Uniinv{\Ut_{31}^{\ast}s_{2i}}}{(\sigma_{\min}(\Sigma_1))^2-\sh_{k_1+i}^2},
\end{equation}
\begin{equation}  \nonumber 
\Uniinv{\Vt_{21i}}\leq \frac{ \sigma_{\min}(\Sigma_1)\Uniinv{\Ut_{31}^{\ast}s_{2i}}+ \sh_{k_1+i}\Uniinv{r_{2i}^T\Vt_{31}}}{(\sigma_{\min}(\Sigma_1))^2-\sh_{k_1+i}^2}. 
\end{equation}
Note that when $k_m\neq k_n$, $(\Ut_{21i},r_{2i})$ or $(\Vt_{21i},s_{2i})$ is empty for large $i$; by taking $\sh_{k_1+i}=0$ for such $i$ the argument carries over. We therefore have for every $i$
\begin{equation}
  \label{eq:ut21i}
\max(\Uniinv{\Ut_{21i}},\Uniinv{\Vt_{21i}})\leq 
\frac{ \max(\|s_{2i}\|,\|r_{2i}\|)
\max(\Uniinv{\Ut_{31}},\Uniinv{\Vt_{31}})
}{\sigma_{\min}(\Sigma_1)-\sh_{k_1+i}}.   
\end{equation}

From~\eqref{eq:Ut31S1S2etc} we also obtain
\begin{equation}
  \label{eq:Ut311final}
\sigma_{\min}(\Sigma_1)\Uniinv{\Ut_{31}}\leq \|A_3\|\Uniinv{\Vt_{31}}+\Uniinv{S_1}+\sum_{i=1}^{k_n-k_1}\|s_{2i}\|\Uniinv{\Vt_{21i}} ,
\end{equation}
and from~\eqref{eq:eq:Ut31S1S2etc2}, 
\begin{equation}
  \label{eq:Vt311final}
\sigma_{\min}(\Sigma_1)\Uniinv{\Vt_{31}}
\leq \Uniinv{\Ut_{31}}\|A_3\|+\Uniinv{R_1}+
\sum_{i=1}^{k_m-k_1}\Uniinv{\Ut_{21i}}\|r_{2i}\|.   
\end{equation}
Hence eliminating $\unii{\Ut_{31}}$ and then $\unii{\Vt_{31}}$, and using~\eqref{eq:ut21i} gives 
\begin{align*}
\max(&\unii{\Ut_{31}},\unii{\Vt_{31}})\\
\leq &
\frac{1}{\Gap}
\left(\max(\uni{S_1},\uni{R_1})+
\sum_{i=1}^{k'}\frac{\max(\|r_{2i}\|,\|s_{2i}\|)^2}{\sigma_{\min}(\Sigma_1)-\sh_{k_1+i}}\max(\Uniinv{\Ut_{31}},\Uniinv{\Vt_{31}})\right).
\end{align*}
Thus by the assumption
$\Gap>\frac{\sum_{i=1}^{k'}\max(\|r_{2i}\|,\|s_{2i}\|)^2}{\sigma_{\min}(\Sigma_1)-\sh_{k_1+i}}$ we have 
\[
\max(\Uniinv{\Ut_{31}},\Uniinv{\Vt_{31}})\leq 
\frac{
\max(\uni{S_1},\uni{R_1})}{\Gap-\frac{\sum_{i=1}^{k'}\max(\|r_{2i}\|,\|s_{2i}\|)^2}{\sigma_{\min}(\Sigma_1)-\sh_{k_1+i}}}.
\]
Finally, the bound~\eqref{eq:sin22subindSVD} follows from combining this with~\eqref{eq:ut21i} and~\eqref{eq:u1andv1}. 
\end{proof}

We note that $R$ or $S$ in the above theorem is allowed to be empty, as in the case where a one-sided projection is employed. This includes the popular randomized SVD algorithm~\cite{halko2011finding}. 
We make two more remarks. 
\begin{remark}[Other approaches for the SVD]
A standard approach to extending results in symmetric eigenvalue problems to the SVD 
is to use the Jordan-Wielandt matrix, for example as in~\cite[Sec.~3]{rcli05}. 
As pointed out in~\cite{nakatsukasa2017accuracy}, 
this has the slight downside of introducing spurious eigenvalues at 0. 
Moreover, the results via Jordan-Wielandt we obtained were less clean and looser than Theorem~\ref{thm:svdmain}. 
Another approach is to work with the Gram matrix $A^*A$, but this unnecessarily squares the singular values and modifies $\gap$ and $\Gap$. 
For these reasons, we have chosen to work directly with the SVD equations.   
\end{remark}

\begin{remark}[Proof of \eqref{eq:boundvecsvd} via Wedin and \cite{nakatsukasa2017accuracy}]
As in Remark~\ref{rem:dksaad}, a proof for~\eqref{eq:boundvecsvd} can be given by combining Wedin's result (the SVD analogue of Davis-Kahan) and~\cite{nakatsukasa2017accuracy} (SVD analogue of Saad's result). 
The sharper bounds~\eqref{eq:boundvecsvd} and~\eqref{eq:sin22subSVD} cannot be obtained this way. 
\end{remark}

\ignore{
\subsubsection{Refined bounds for singular subspaces}As in Theorem~\ref{thm:eigrefine}, we can refine the bounds by using individual information about $\|R_i\|$. 
\begin{theorem}\label{thm:svdrefined}
In the setting of Theorem~\ref{thm:svdmain}, 
$\max(\|\sin\angle(U_1,\widehat  U_1)\|,\|\sin\angle(V_1,\widehat  V_1)\|)$ is bounded from above by 
\begin{equation}  \label{eq:sin22subSVD}
 \frac{\max(\|R_1\|,\|S_1\|)}{\Gap-
\frac{\max(\|R_2\|,\|\St_2\|)^2}{\gap}}\sqrt{1+\frac{\max(\|R_2\|,\|\St_2\|)^2}{\gap^2}}, 
\end{equation}
and 
\begin{equation}
  \label{eq:sin22subindSVD}
 \frac{\max(\|R_1\|,\|S_1\|)}{
\Gap-\max\left(\sum_{i=2}^{k}\frac{\|R_{i}\|^2}{|\sigma  -\sh_{i}|}, 
\sum_{i=2}^{k}\frac{\|S_{i}\|^2}{|\sigma  -\sh_{i}|}
\right)
}\sqrt{1+
\max\left(\sum_{i=2}^{k}\frac{\|R_{i}\|}{|\sigma  -\sh_{i}|},
\sum_{i=2}^{k}\frac{\|S_{i}\|}{|\sigma  -\sh_{i}|}
\right)^2}.
\end{equation}
\end{theorem}

We omit the proof as it is long but not enlightening, which combines techniques in the proofs of Theorems~\ref{thm:svdmain} and \ref{thm:eigrefine}. 
}

\section{Eigenvectors of a self-adjoint operator}\label{sec:hilbert}
So far we have specialized to finite-dimensional matrices as the analysis is 
elementary and the situation is more transparent. In this final section, as in~\cite{knyazevmc97,ovtchinnikov2006cluster}, we extend the discussion to the infinite-dimensional case, where 
the matrix is generalized to a 
self-adjoint operator $\A:\mathcal{H}\rightarrow \mathcal{H}$ 
on a Hilbert space $\mathcal{H}$ with inner product $\left<\cdot,\cdot\right>$. 
Unlike the previous studies, which assumed the operators are bounded, our discussion allows $\A$ to be unbounded, thus is applicable for example to differential operators $\mathcal{A}u=u''$; 
in this case, we assume that $\A$ is densely defined, as is customary. 

Let $Q$ be a subspace of $\mathcal{H}$, which 
is of finite dimension $k$ with orthonormal basis $q_1,\ldots,q_k$. 
In the Rayleigh-Ritz process for $\A$, we compute the $k\times k$ matrix $A_1$ with $(i,j)$ element $\left<q_i,\A q_j\right>$ and its eigenvalue decomposition 
$A_1 = \Omega\mbox{diag}(\lh_1,\ldots,\lh_k) \Omega^\ast$
to obtain the Ritz values $\lh_1,\ldots,\lh_k$ and Ritz vectors $[q_1,\ldots,q_k]\Omega$. 
Denote by $Q_1,Q_2\in\mathcal{H}$ the resulting Ritz subspaces corresponding to disjoint sets of eigenvalues of $A_1$ (we have $Q=Q_1\oplus Q_2$), and let $Q_3$ be the (infinite-dimensional) orthogonal complement of $Q$ such that $\mathcal{H}=Q_1\oplus Q_2\oplus Q_3$ is an orthogonal direct sum. 

For simplicity, we treat the case where $Q_1$ is one-dimensional
(subspace versions can be obtained, generalizing Section~\ref{sec:RRangles}). 
That is, let $(\lh,\uh)$ be a Ritz pair with $\uh=q_1$, and 
suppose that $\mathcal{A}u = \lambda u$; note that this is an assumption, as a self-adjoint operator may not have any eigenvalue (e.g.~\cite[Ch.~9]{hunter}), although the spectrum is always nonempty. 
 The goal is to 
bound $\sin\angle(\uh,u)$. 

Denote by $\P_i$ be the orthogonal projectors onto each subspace $Q_i$. 
We define $\mathcal{A}_{ij}:=\P_i\A \P_j$. 
Then the R-R process forces $\mathcal{A}_{12}=0, \mathcal{A}_{21}=0$. 
Note that $\mathcal{A}_{13}^*=\mathcal{A}_{31}, \mathcal{A}_{23}^*=\mathcal{A}_{32}$ (where $*$ denotes the adjoint of the operators), and these terms represent the residuals, hence we write 
$\|R_1\|=\|R_{31}\|$ and $\|R\|=\|\mathcal{A}_{31}+\mathcal{A}_{32}\|$.
Also define 
$\|R_2\|=\|\mathcal{A}_{32}\| (=\|\mathcal{A}_{23}\|)$, and 
$\|r_i\|=\|\mathcal{A}_{32}\mathcal{P}_{2,i}\|(=\|\mathcal{P}_{2,i}\mathcal{A}_{23}\|)$ for $i=2,\ldots,k$, where 
$\mathcal{P}_{2,i}$ is the $1$-dimensional projection onto the $i$th Ritz vector. 
The quantities $\gap$ and $\Gap$ are defined by 
$\gap=\min|\lh-\lambda(\mathcal{A}_{22})|$, 
$\Gap=\min|\lh-\lambda(\mathcal{A}_{33})|$, in which  $\lambda(\mathcal{A}_{ii})$ denotes the spectrum of the restriction of $\mathcal{A}_{ii}$ to $Q_i$.



\begin{theorem}\label{thm:hilbert}
Under the above assumptions and notation, 
\begin{equation}  \label{eq:boundvecselfadjoint}
\sin\angle(u,\uh)\leq 
\frac{\|R\|}{\Gap}\sqrt{1+\frac{\|R_2\|^2}{\gap^2}}
\quad
\left(
\leq  \frac{\|R\|}{\Gap}(1+\frac{\|R_2\|}{\gap})
\right).
\end{equation}
Moreover, 
if $\Gap>\frac{\|R_2\|^2}{\gap}$, then 
\begin{equation}
  \label{eq:sin22_inf}
\sin\angle(u,\uh) 
\leq \frac{\|R_1\|}{\Gap-\frac{\|R_2\|^2}{\gap}}\sqrt{1+\frac{\|R_2\|^2}{\gap^2}}  , 
\end{equation}
and if $\Gap>\sum_{i=2}^{k}\frac{\|r_{i}\|^2}{|\lambda  -\lh_{i}|}$, then 
\begin{equation}
  \label{eq:sin2indiv_inf}
\sin\angle(u,\uh) 
\leq 
\frac{\|R_1\|}{
\Gap-\sum_{i=2}^{k}\frac{\|r_{i}\|^2}{|\lambda  -\lh_{i}|}}
\sqrt{1+
\left(\sum_{i=2}^{k}\frac{\|r_{i}\|}{|\lambda  -\lh_{i}|}\right)^2}. 
\end{equation}
\end{theorem}

\begin{proof}  
Writing $u=\sum_{i=1}^3 u_i$ with $u_i\in Q_i$, 
the $Q_i$-component of $\mathcal{A}u=\lambda u$ each implies 
\begin{subequations}  
  \begin{align}
\lambda u_1&=\mathcal{A}_{11}u_1+ \mathcal{A}_{13}u_3 , \label{eq:Aihilba}\\
\lambda u_2&=\mathcal{A}_{22}u_2+ \mathcal{A}_{23}u_3 , \label{eq:Aihilbb}\\
\lambda u_3&=\mathcal{A}_{31}u_1+ \mathcal{A}_{32}u_2+\mathcal{A}_{33}u_3 .\label{eq:Aihilbc}
  \end{align}
\end{subequations}
Our goal is to bound $\sin\angle(u,\uh)=\sqrt{\|u_2\|^2+\|u_3\|^2}$. 

We first derive~\eqref{eq:boundvecselfadjoint}, an analogue of Theorem~\ref{thm:mainscalar}. 
By~\eqref{eq:Aihilbc}, we have 
\[\|(\mathcal{A}_{33}-\lambda)u_3\|= \|\mathcal{A}_{31}u_1+\mathcal{A}_{32}u_2\|=\|(\mathcal{A}_{31}+\mathcal{A}_{32})u\|\leq \|\mathcal{A}_{31}+\mathcal{A}_{32}\|=\|R\|.\] 
Together with the fact $\|(\mathcal{A}_{33}-\lambda)u_3\|\geq \Gap \|u_3\|$ 
(\cite[\S~V.3.5]{kato1995perturbation}; to see this, note that $v=(\mathcal{A}_{33}-\lambda)u_3$ implies $(\mathcal{A}_{33}-\lambda)^{-1}v=u_3$, hence $\|u_3\|\leq \|(\mathcal{A}_{33}-\lambda)^{-1}\|\|v\|\leq \Gap \|v\|$), 
we obtain 
\begin{equation}  \label{eq:u3}
\|u_3\|\leq \frac{\|R\|}{\Gap}  . 
\end{equation}

Now~\eqref{eq:Aihilbb} gives 
$(\mathcal{A}_{22}-\lambda)u_2=- \mathcal{A}_{23}u_3 $. Since $\|\mathcal{A}_{23}\|=\|\mathcal{A}_{23}^*\|=\|\mathcal{A}_{32}\|= \|R_2\|$, and  $\|(\mathcal{A}_{22}-\lambda)u_2\|\geq \gap \|u_2\|$, 
we thus have 
$\|u_2\|\leq \frac{\|R_2\|}{\gap} \|u_3\|$.
Using this and~\eqref{eq:u3}, 
we obtain~\eqref{eq:boundvecselfadjoint}. 

We now turn to~\eqref{eq:sin2indiv_inf}; the proof of~\eqref{eq:sin22_inf} is similar and omitted. 
As in Theorem~\ref{thm:mainsubspace}, the idea is to improve the estimate of $\|u_2\|$ using \eqref{eq:Aihilbb}. 
Projecting it onto $\mathcal{P}_{2,i}$ gives 
$\mathcal{P}_{2,i}(\mathcal{A}_{22}-\lambda)u_2+ \mathcal{P}_{2,i}\mathcal{A}_{23}u_3 =0$
for $i=2,\ldots,k$, and 
by assumption $\mathcal{P}_{2,i}\mathcal{A}_{22}=\lh_i\mathcal{P}_{2,i}$, so 
$(\lh_i-\lambda)\mathcal{P}_{2,i}u_2+ \mathcal{P}_{2,i}\mathcal{A}_{23}u_3 =0, $
hence
\begin{equation}
  \label{eq:Piu2}
\|\mathcal{P}_{2,i}u_2\|\leq \frac{\|\mathcal{P}_{2,i}\mathcal{A}_{23}u_3\|}{|\lh_i-\lambda|}
\leq \frac{\|\mathcal{P}_{2,i}\mathcal{A}_{23}\|\|u_3\|}{|\lh_i-\lambda|}
= \frac{\|r_i\|\|u_3\|}{|\lh_i-\lambda|}, 
\end{equation}
where we used  $\|r_i\|=\|\mathcal{P}_{2,i}\mathcal{A}_{23}\|$ for the final equality. 
The inequality~\eqref{eq:Piu2} holds for $i=2,\ldots,k$. 

Now 
since $\mathcal{A}_{32} = 
 \mathcal{A}_{32}\mathcal{P}_{2}=
\sum_{i=2}^k\mathcal{A}_{32}\mathcal{P}_{2,i} $, 
we can rewrite~\eqref{eq:Aihilbc} as 
$\mathcal{A}_{31}u_1+ \sum_{i=2}^k \mathcal{A}_{32}\mathcal{P}_{2,i}u_2=-(\mathcal{A}_{33}-\lambda)u_3$. Therefore 
\begin{align*}
\|(\mathcal{A}_{33}-\lambda)u_3\|&\leq 
\|\mathcal{A}_{31}u_1\|+ \sum_{i=2}^k\|\mathcal{A}_{32}\mathcal{P}_{2,i} u_2\|\leq \|R_1\|+ \sum_{i=2}^k\|\mathcal{A}_{32}\mathcal{P}_{2,i}\|\|\mathcal{P}_{2,i} u_2\|  \\
&\leq \|R_1\|+ \sum_{i=2}^k\frac{\|r_i\|^2\|u_3\|}{|\lh_i-\lambda|}  , 
\end{align*}
where we used~\eqref{eq:Piu2} and $\|r_i\|=\|\mathcal{A}_{32}\mathcal{P}_{2,i}\|$. 
Together with $\|(\mathcal{A}_{33}-\lambda)u_3\|\geq \Gap\|u_3\|$ we obtain 
\begin{equation}  \label{eq:u3bound}
\|u_3\|\leq \frac{\|R_1\|}{\Gap\left(1-\sum_{i=2}^k\frac{\|r_i\|^2}{|\lh_i-\lambda|} \right)}   . 
\end{equation}

Finally,~\eqref{eq:Piu2} together with 
$\|u_2\|^2=\sum_{i=2}^k\|\mathcal{P}_{2,i}u_2\|^2$ gives
$\|u_2\|^2\leq \left(\sum_{i=2}^k\frac{\|r_i\|}{|\lh_i-\lambda|}\right)^2 \|u_3\|^2$, 
so 
\[
\sin\angle(u,\ut) = \sqrt{\|u_2\|^2+\|u_3\|^2}
\leq \frac{\|R_1\|}{\Gap\left(1-\sum_{i=2}^k\frac{\|r_i\|^2}{|\lh_i-\lambda|} \right)}\sqrt{1+\left(\sum_{i=2}^k\frac{\|r_i\|}{|\lh_i-\lambda|}\right)^2 }, 
\]
completing the proof of~\eqref{eq:sin2indiv_inf}. 
\end{proof}

\subsection{Experiments: Sturm-Liouville eigenvalue problem}\label{sec:SL}

We illustrate Theorem~\ref{thm:hilbert} with a 
simple Sturm-Liouville eigenvalue problem (e.g.~\cite[\S~3.5]{folland1992fourier})
\begin{equation}  \label{eq:SL}
\mathcal{A}u = u''=\lambda u,\qquad u'(0)  = \alpha u(0), 
\quad u'(\pi) = \beta u(\pi), 
\qquad u\in\mathcal{H}=H^{2}(0,\pi).
\end{equation}
$\mathcal{A}$ is an unbounded self-adjoint operator, with a full set of (infinitely many) orthonormal eigenfunctions. 
Here we take $\alpha = 1, \beta = -1$. 
The exact eigenvalues are $\lambda_i=-\nu_i^2$, where 
$\nu_i$ are the solutions for $\tan\pi\nu = 2\nu/(\nu^2-1)$, with corresponding eigenfunction $\nu_i\cos\nu_i x+\alpha\sin\nu_i x$~\cite[\S~3.5]{folland1992fourier}. 
We attempt to compute the eigenpairs
 with the smoothest eigenfunctions, i.e., eigenpairs closest to 0. 
To do this, a natural idea is to take low-degree polynomials. We take the trial subspace to be the $k$-dimensional subspace of polynomials $p$ of degree up to $k+1$ that satisfy the two boundary conditions 
$p'(0)  = \alpha p(0)$ and $p'(\pi) = \beta p(\pi)$. 
Figure~\ref{fig:SRFolland} (left) shows the basis functions obtained in this way, for $k=7$. Such computations can be done conveniently using Chebfun~\cite{chebfunguide}. 

Having defined the subspace $Q$, we can perform R-R to obtain the Ritz vectors (which are functions in $\mathcal{H}$ here), along with the Ritz values. 

Figure~\ref{fig:SRFolland} shows the convergence of 
$\angle(u,\uh)$  to the eigenfunction $u$ for the smallest eigenpair and its bounds, analogous to Figure~\ref{fig:lob}. As in that experiment, our bound~\eqref{eq:sin22_inf} gives tighter bounds for the actual error, although here Davis-Kahan also performs well, since $\gap$ is not very small. 


\begin{figure}[htpb]
  \begin{minipage}[t]{0.49\hsize}
      \includegraphics[width=0.9\textwidth]{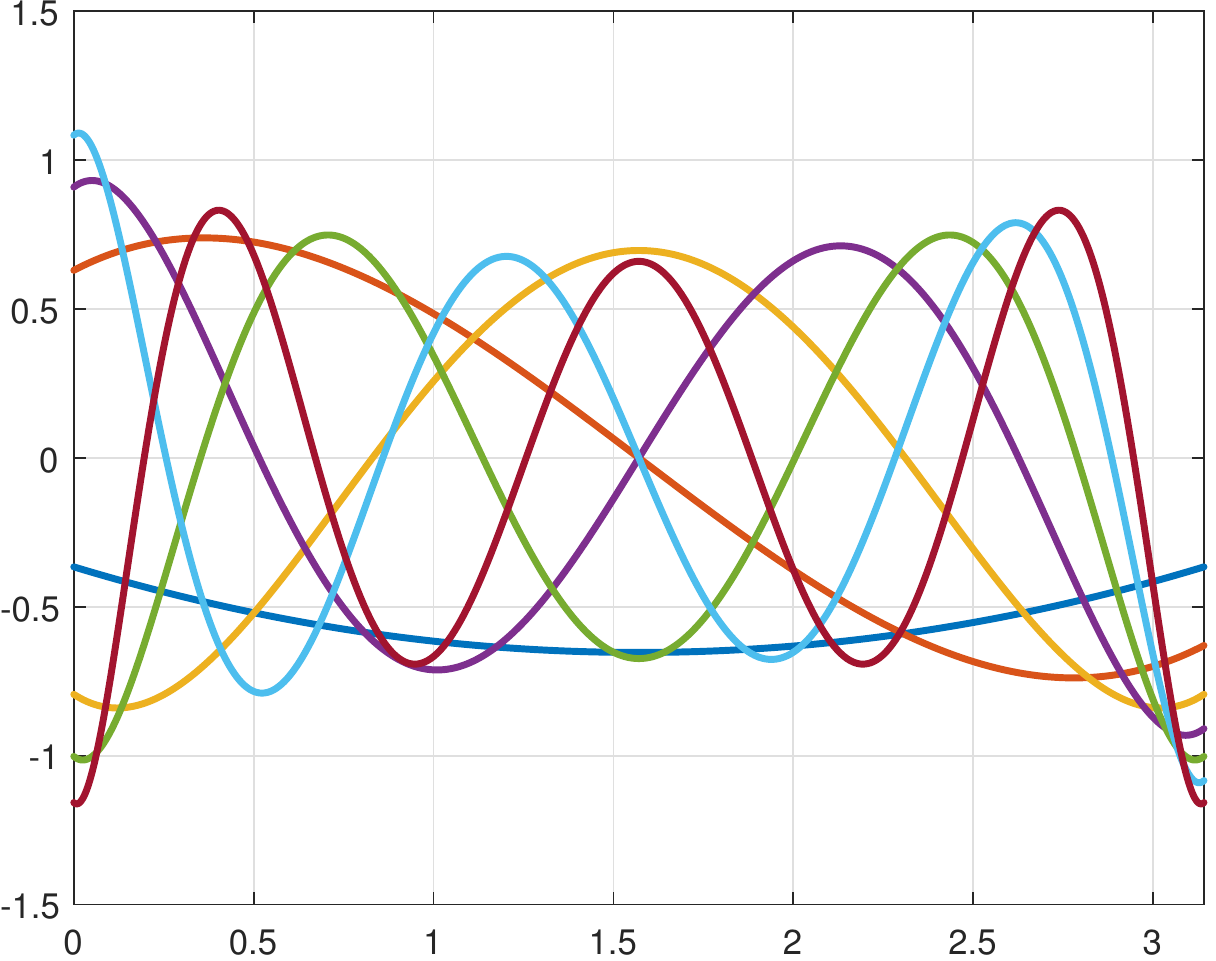}
  \end{minipage}   
  \begin{minipage}[t]{0.49\hsize}
      \includegraphics[width=1.0\textwidth]{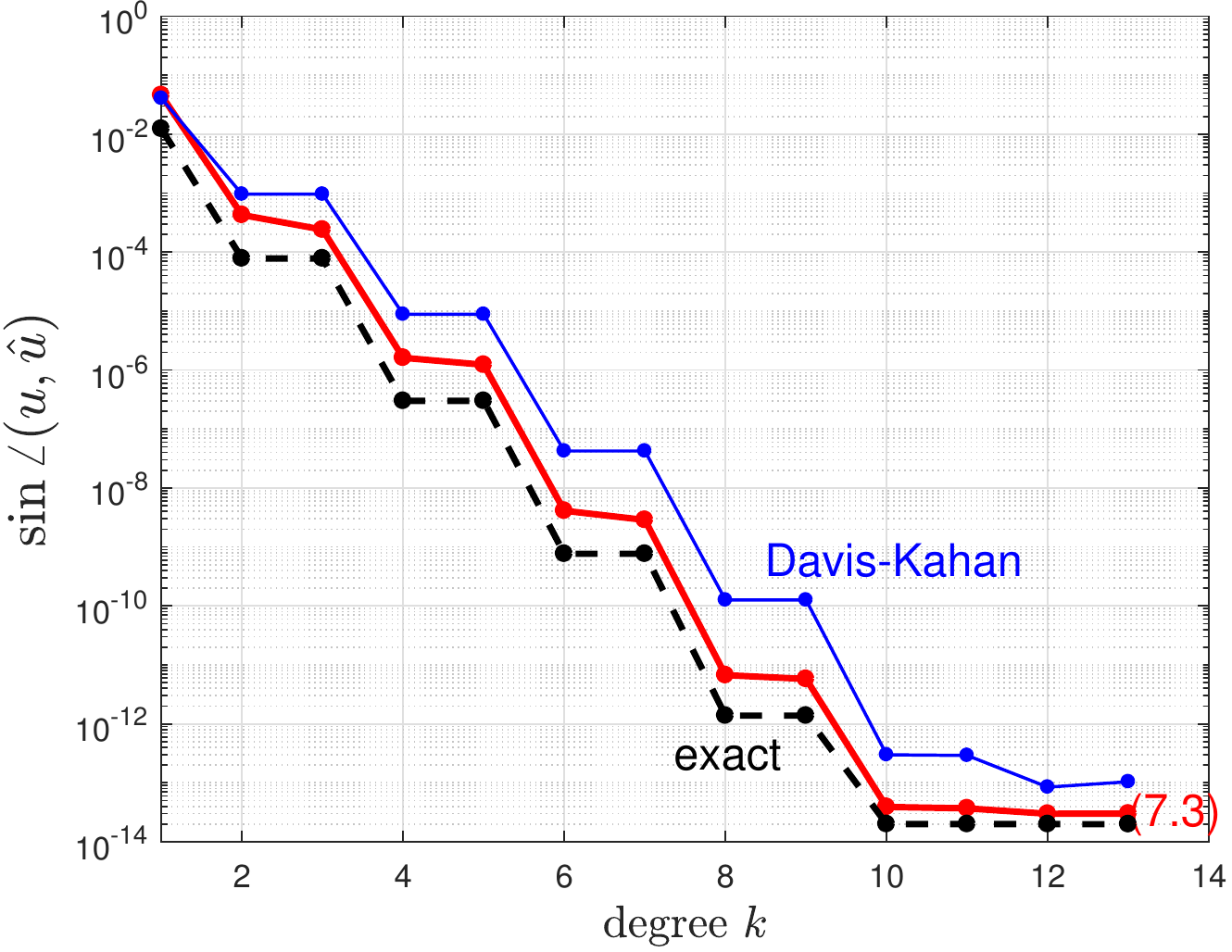}
  \end{minipage}
  \caption{
Left: Basis functions for projection subspace $Q$, satisfying $u'(0)=u(0),u'(\pi)=-u(\pi)$. 
Right: Convergence of $\angle(u,\uh)$ and its bounds. 
}
  \label{fig:SRFolland}
\end{figure}

Finally, in Figure~\ref{fig:SRFolland_res} we illustrate the behavior of the residual function $\mathcal{A}\uh-\lh \uh$ as $k$ varies. 
Note that $\uh$ is determined up to a sign flip $\pm 1$; 
here we chose $\uh(1)>0$. 
We make two observations. First, evidently the norm 
$\|\mathcal{A}\uh-\lh \uh\|$ decays rapidly as $k$ increases, essentially like the right plot in Figure~\ref{fig:SRFolland}. The second and more interesting observation is that the residuals appear to become more and more oscillatory (non-smooth) as $k$ grows. This is a typical phenomenon, and can be explained as follows. As emphasized repeatedly in this paper, R-R forces the residual to be orthogonal to $Q$, which contains the ``smoothest'' functions. Consequently, in the Legendre expansion of the residual $\mathcal{A}\uh-\lh \uh=\sum_{i=0}^\infty c_iP_i(x)$, $|c_i|$ are small for $i< k$; they are bounded roughly by $\|u_2\|$, which is $O(\|\mathcal{A}\uh-\lh\uh\|^2)$ by~\eqref{eq:Piu2}. This also reflects the main result in~\cite{knyazevmc97}; recall Remark~\ref{rem:dksaad}. 
By growing $k$, the residual becomes orthogonal to more and more of these smoothest functions, and therefore 
 becomes more oscillatory. 

\begin{figure}[htpb]
  \begin{center}
      \includegraphics[width=1.0\textwidth]{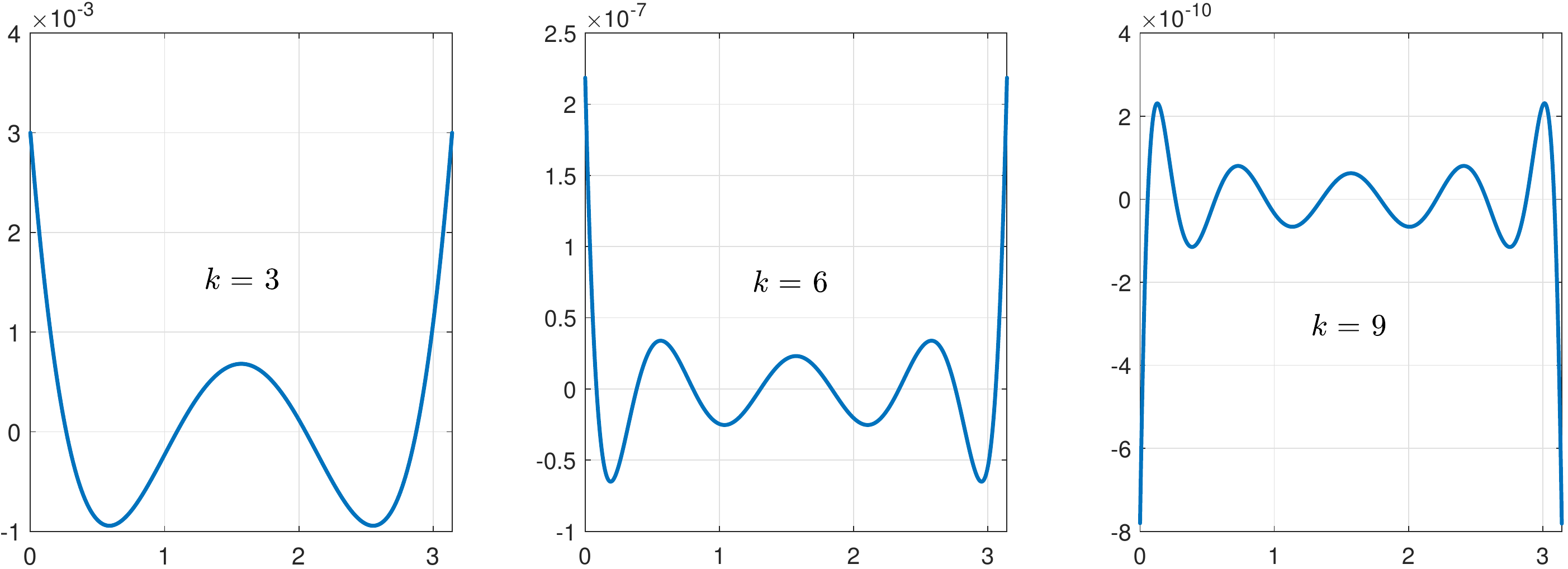}
  \end{center}
  \caption{
Residual function $\mathcal{A}\uh-\lh \uh$ for $k=3, 6$ and $9$. 
}
  \label{fig:SRFolland_res}
\end{figure}

\subsection*{Acknowledgements}
I am grateful to Andrew Knyazev for helpful discussions and bringing~\cite{ovtchinnikov2006cluster} to my attention, and Mayuko Yamashita for the help in Section~\ref{sec:hilbert}. 

\def\noopsort#1{}\def\l{\char32l}\def\v#1{{\accent20 #1}}
  \let\^^_=\v\def\hbk{hardback}\def\pbk{paperback}

\end{document}